\documentclass[a4paper,12pt]{amsart}
\usepackage{amssymb}
\usepackage{enumitem}
\usepackage[alphabetic]{amsrefs}
\usepackage{tikz}
\usetikzlibrary{graphs,automata,quotes}
\usepackage{blindtext, xcolor}
\usepackage{comment}
\usepackage{mathrsfs}
\usepackage{geometry}
\usepackage{a4wide}
\usepackage{fullpage}
\usepackage{mathtools}
\usepackage{subcaption} 
\usepackage[classfont=sanserif,
            langfont=caps,
            funcfont=typewriter,
            basic,
            small]{complexity} 

\usepackage[colorlinks]{hyperref}
\usepackage{cleveref}

\usepackage{listings}
\usepackage{pythonhighlight}

\definecolor{dkgreen}{rgb}{0,0.6,0}
\definecolor{gray}{rgb}{0.5,0.5,0.5}
\definecolor{mauve}{rgb}{0.58,0,0.82}

\makeatletter
\@namedef{subjclassname@2020}{%
  \textup{2020} Mathematics Subject Classification: 5C31 primary; 82B20, 68W25 secondary} 
\makeatother

\newtheorem{corollary}{Corollary}[section]
\newtheorem{theorem}[corollary]{Theorem}
\newtheorem{lemma}[corollary]{Lemma}
\newtheorem{proposition}[corollary]{Proposition}

\newtheorem{definition}[corollary]{Definition}

\newtheorem{remark}[corollary]{Remark}

\newtheorem*{theorem*}{Theorem}
\newtheorem*{corollary*}{Corollary}

\numberwithin{equation}{section}

\newcommand{\SP}{S}
\let\series\bowtie

\newcommand{\Hs}{\mathcal{H}}

\newcommand{\N}{\mathbb{N}}
\newcommand{\Q}{\mathbb{Q}}
\renewcommand{\R}{\mathbb{R}}
\renewcommand{\C}{\mathbb{C}}
\newcommand{\Z}{\mathcal{Z}}
\newcommand{\A}{\mathcal{A}}

\newcommand{\Cext}{\hat{\C}}

\newcommand{\YGz}{y_{G_0}}
\newcommand{\airquotes}[1]{``#1''}
\newcommand{\pref}[1]{\textbf{\hyperref[#1]{here}}}

\DeclarePairedDelimiter{\abs}{\lvert}{\rvert}
\newfunc{\size}{size}

\renewcommand{\phi}{\varphi}

\renewcommand{\epsilon}{\varepsilon}

\usepackage{graphicx} 

\title{On the complex zeros and the computational complexity of approximating the reliability polynomial}
\author[F. Bencs]{Ferenc Bencs}
\address{Ferenc Bencs, Centrum Wiskunde \& Informatica, P.O. Box 94079 1090 GB Amsterdam, The Netherlands}
\email{\texttt{ferenc.bencs@gmail.com}}
\thanks{FB was funded by the Netherlands Organisation of Scientific Research (NWO): VI.Veni.222.303}
\author[C. Piombi]{Chiara Piombi}
\address{Chiara Piombi, Alma Mater Studiorum - Universit\`a di Bologna - Via Zamboni, 33 - 40126 Bologna - Italy} 
\email{\texttt{piombichiara@gmail.com}}
\author[G.Regts]{Guus Regts}
\address{Guus Regts, Korteweg de Vries Institute for Mathematics, University of Amsterdam. P.O. Box 94248  
1090 GE Amsterdam The Netherlands}
\email{\texttt{g.regts@uva.nl}}
\thanks{GR was funded by the Netherlands Organisation of Scientific Research (NWO): VI.Vidi.193.068}
\date{\today}

\begin{document}

\begin{abstract}
In this paper we relate the location of the complex zeros of the reliability polynomial to parameters at which a certain family of rational functions derived from the reliability polynomial exhibits chaotic behaviour. 
We use this connection to prove new results about the location of reliability zeros. In particular we show that there are zeros with modulus larger than $1$ with essentially any possible argument.
We moreover use this connection to show that approximately evaluating the reliability polynomial for planar graphs at a non-positive algebraic number in the unit disk is \textsc{\#P}-hard.
\end{abstract}
\maketitle
\section{Introduction}
Consider for a connected (multi)graph\footnote{Throughout this paper we will allow our graphs to have parallel edges.} $G=(V,E)$ the probability, $R(G;p)$, that it remains connected if every edge has a probability $p\in [0,1]$ to fail.
The quantity $R(G;p)$ is in fact a polynomial in the failure probability $p$:
\begin{equation}
R(G;p)=\sum_{\substack{A\subseteq E\\(V,A)\text{ connected}}}(1-p)^{|A|}p^{|E|-|A|},\label{eq:def R(G;p)}    
\end{equation}
known as the \emph{(all-terminal) reliability polynomial}.
The study of graph reliability started during the Cold War as a model for communication networks where nodes and links could fail because of power outages, sabotage or being destroyed by bombs: there are references to lectures about the subject as early as 1952~\cite{vonneumann-1956}, though the reliability polynomial was first explicitly defined by Moore and Shannon in 1956~\cite{moore-1956}.

Since $R(G;p)$ is a polynomial in $p$ one can ask about the location of its complex zeros, henceforth called \emph{reliability zeros}, which gives rise to intriguing questions.
For example Brown and Colbourn~\cite{brown-1992} conjectured in 1992 that all reliability zeros are contained in the unit disk, but about twelve years later Royle and Sokal~\cite{royle-2004} found examples of reliability zeros barely outside the unit disk thereby disproving the conjecture. 
Later Brown and Mol~\cite{brown-2017} managed to find reliability zeros of slightly bigger modulus. However it remains open whether reliability zeros are uniformly bounded or not.

In the present paper we are interested in the relation between the location of the reliability zeros and the computational complexity of approximately computing $R(G;p)$ for a given algebraic number $p\in \C\setminus\{0,1\}$.
Let us first mention, in regard to \emph{exact} computation, that it has been known for about forty years that exactly computing the value $R(G;p)$ for a rational number $p\in (0,1)$ is \textsc{\#P}-hard by work of Ball and Provan~\cite{reliabilityishard}. Vertigan~\cite{vertigan-2005} showed that this extends to any algebraic number $p\in \C\setminus \{0,1\}$ even when the graphs are restricted to be planar. 

Our motivation stems from a recent line of work that relates the presence of zeros of graph polynomials and the computational hardness of approximately computing evaluations of these polynomials such as for the independence polynomial~\cites{Bezakovahardcore,chaoticratios} the partition function of the Ising model~\cite{buys-2022} the matching polynomial~\cite{Bezakovamatching} and the Tutte polynomial~\cites{galanis-2022,main-roots,main-approx}.
These papers moreover connected the presence of zeros and computational hardness to chaotic behaviour of a family of naturally associated rational functions.
Our main contribution is to uncover a like connection in the setting of the reliability polynomial and use this connection to prove new results about the location of the reliability zeros and to show that even approximately evaluating the reliability polynomial is \textsc{\#P}-hard in many cases.

\subsection{Our contributions}
To state our contributions we need a few definitions. 
Let $G=(V,E)$ be a graph and let $s,t\in V$ two distinct vertices, called the \emph{source} and \emph{sink} respectively. 
We call the triple $(G,s,t)$ a \emph{two-terminal} graph and denote the collection of all two-terminal graphs by $\mathcal{G}_2$.
Following the notation used in~\cite{brown-2017}, we define a \emph{$s-t$ split} in $G$ to be a spanning subgraph such that from each vertex $v$ of $G$ there is either a path from $v$ to $s$ or a path from $v$ to $t$, but not both. 
We then define the \emph{split reliability polynomial} as:
\begin{equation}
\SP(G;p)=\sum_{\substack{A\subseteq E\\(V,A)\text{ $s-t$ split}}}(1-p)^{|A|}p^{|E|-|A|}.
\end{equation}
Alternatively, let $\hat G$ be the graph obtained from $G$ by merging the vertices $s$ and $t$ (and any edge between $s$ and  $t$ becoming a loop).
Then $\SP(G;p)=R(\hat{G};p)-R(G;p)$.
Next we define the \emph{effective edge interaction} of $G$ at $p$ as
\begin{equation}\label{eq:def effective}
y_G(p) := (1-p)\frac{\SP(G;p)}{R(G;p)}+1
\end{equation}
and the \emph{virtual edge interaction} of $G$ at $p$ as
\begin{equation}\label{eq:def virtual}
\hat{y}_G(p) := \frac{R(G;p)}{\SP(G;p)}+1.
\end{equation}
We will motivate the terminology later, for now it suffices to think of $y_G(p)$ and $\hat{y}_G(p)$ as a rational functions in $p$. Note that we have omitted $s,t$ from the notation.

We need to introduce series and parallel compositions of two-terminal graphs.
Let $G_1$ and $G_2$ be two-terminal graphs with sources $s_1,s_2$ and sinks $t_1,t_2$ respectively. The \emph{parallel composition} of $G_1$ and $G_2$ (denoted $G_1\parallel G_2$) is the graph obtained from the disjoint union of $G_1$ and $G_2$ by identifying $s_1$ and $s_2$, and $t_1$ and $t_2$ into single vertices, respectively the source and sink of $G_1\parallel G_2$. The \emph{series composition} of $G_1$ and $G_2$ (denoted $G_1\series G_2$) is the graph obtained from the disjoint union of $G_1$ and $G_2$ by identifying $t_1$ and $s_2$ into a single vertex and with source $s_1$ and sink $t_2$.

\begin{center}
    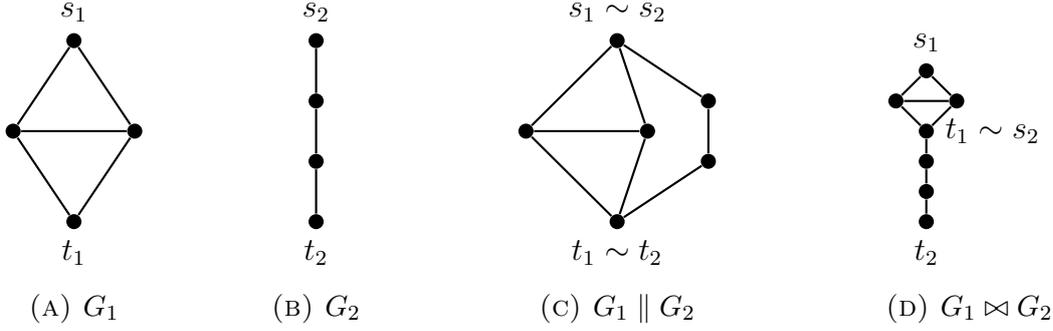
\begin{figure}[ht]

\begin{subfigure}{0.19\textwidth}
\centering
\begin{tikzpicture}[scale=0.4]
    \node[circle,fill,inner sep=2pt,label=above:$s_1$] (s) at (2,6) {};
    \node[circle,fill,inner sep=2pt] (0) at (0,3) {};
    \node[circle,fill,inner sep=2pt] (1) at (4,3) {};
    \node[circle,fill,inner sep=2pt,label=below:$t_1$] (t) at (2,0) {};
    \draw[thick] (s) -- (0) -- (t);
    \draw[thick] (s) -- (1) -- (t);
    \draw[thick] (0) -- (1);
\end{tikzpicture}
    \caption{$G_1$}
\end{subfigure}
\hfill
\begin{subfigure}{0.19\textwidth}
\centering
    \begin{tikzpicture}[scale=0.4]
    \node[circle,fill,inner sep=2pt,label=above:$s_2$] (s) at (2,6) {};
    \node[circle,fill,inner sep=2pt] (0) at (2,4) {};
    \node[circle,fill,inner sep=2pt] (1) at (2,2) {};
    \node[circle,fill,inner sep=2pt,label=below:$t_2$] (t) at (2,0) {};
    \draw[thick] (s) -- (0) -- (1) -- (t);
\end{tikzpicture}
    \caption{$G_2$}
\end{subfigure}
\hfill
\begin{subfigure}{0.29\textwidth}
\centering
    \begin{tikzpicture}[scale=0.4]
    \node[circle,fill,inner sep=2pt,label=above:$s_1\sim s_2$] (s) at (3,6) {};
    \node[circle,fill,inner sep=2pt] (0) at (0,3) {};
    \node[circle,fill,inner sep=2pt] (1) at (4,3) {};
    \node[circle,fill,inner sep=2pt,label=below:$t_1\sim t_2$] (t) at (3,0) {};
    \node[circle,fill,inner sep=2pt] (2) at (6,4) {};
    \node[circle,fill,inner sep=2pt] (3) at (6,2) {};
    \draw[thick] (s) -- (0) -- (t);
    \draw[thick] (s) -- (1) -- (t);
    \draw[thick] (0) -- (1);
    \draw[thick] (s) -- (2) -- (3) -- (t);
\end{tikzpicture}
    \caption{$G_1\parallel G_2$}
\end{subfigure}
\begin{subfigure}{0.29\textwidth}
\centering
    \begin{tikzpicture}[scale=0.4]
    \node[circle,fill,inner sep=2pt,label=above:$s_1$] (s) at (1,2) {};
    \node[circle,fill,inner sep=2pt] (0) at (0,1) {};
    \node[circle,fill,inner sep=2pt] (1) at (2,1) {};
    \node[circle,fill,inner sep=2pt, label=right:$t_1\sim s_2$] (m) at (1,0) {};
    \node[circle,fill,inner sep=2pt] (2) at (1,-1) {};
    \node[circle,fill,inner sep=2pt] (3) at (1,-2) {};
    \node[circle,fill,inner sep=2pt,label=below:$t_2$] (t) at (1,-3) {};
    \draw[thick] (s) -- (0) -- (m);
    \draw[thick] (s) -- (1) -- (m);
    \draw[thick] (0) -- (1);
    \draw[thick] (m) -- (2) -- (3) -- (t);
\end{tikzpicture}
    \caption{$G_1\series G_2$}
\end{subfigure}
 \caption{An example of series and parallel composition.}
       
\end{figure}

\end{center}

A two-terminal graph $G$ is called \emph{series-parallel} if it can be obtained
from series and parallel compositions of a single edge. 
For a two terminal graph $G$ we denote by $G^T$ the two-terminal graph obtained from $G$ by flipping the role of the source and the sink.
Let $G$ be a two-terminal graph; we will define the collection of \emph{series-parallel graphs generated by} $G$ as the collection of all two-terminal graphs obtained from series-parallel composition starting with $G$ or $G^T$ and denote this by $\Hs_G$.
For example, $\Hs_{K_2}$ denotes the set of all series-parallel graphs.
Note that a simple induction argument shows that $\Hs_G$ is closed under the operation $H\mapsto H^T$.
Next we define some subsets of the complex numbers that we want to study.
Let $G_0$ be a two-terminal graph. 
We define its \emph{exceptional set} by 
\begin{align*}
    \mathcal{E}(G_0):=\{p\in \C\mid R(G_0;p)=-\SP(G_0;p)\}.
\end{align*}

We define
\begin{align}
     \mathcal{Z}_{G_0}:=&\{p\in \C \setminus \mathcal{E}(G_0)\mid R(G;p)=0 \text{ for some }G\in \mathcal{H}_{G_0}\}\label{eq:zero-locus};
     \\
    \mathcal{D}_{G_0}:=&\{p\in \C\setminus \mathcal{E}(G_0)\}\mid \{y_G(p)\mid G\in \mathcal{H}_{G_0}, R(G;p)\neq 0\} \text{ is dense in  }\C\}\label{eq:density-locus};
    \\
     \mathcal{A}_{G_0}:=&\{p\in \C\setminus \mathcal{E}(G_0)\mid 1<|\hat{y}_G(p)|<\infty, \text{ and } \hat{y}_G(p)\notin \mathbb{R} \text{ for some } G\in \mathcal{H}_{G_0}\}\label{eq:activity-locus};
\end{align}
and refer to these sets as the \emph{zero-locus}, \emph{density locus} and \emph{activity locus} of $\mathcal{H}_{G_0}$ respectively. 
We moreover define the real analogues of these loci, restricting $p$ to be real, and denote these by $\mathcal{Z}_{G_0}^{\mathbb{R}}, \mathcal{D}_{G_0}^{\mathbb{R}}$ and $\mathcal{A}_{G_0}^{\mathbb{R}}$ respectively, 
where for the real density locus we require the set to be dense in $\mathbb{R}$ and for the real activity locus we also require that $\hat{y}_{G}(p)<-1$ and of course disregard the requirement that $\hat{y}_{G}(p)\notin \mathbb{R}$.

Our first main result relates the different loci defined above and will be proved in Section~\ref{sec:proof main}.
\begin{theorem}\label{thm:main equal}
Let $G_0$ be a two-terminal graph. Then the closure of the activity-locus of $\mathcal{H}_{G_0}$ equals the closure of the zero-locus of $\mathcal{H}_{G_0}$ and the closure of the density-locus of $\mathcal{H}_{G_0}$. More precisely, we have  
\begin{align*}
\overline{\Z_{G_0}}=\overline{\A_{G_0}} \quad \text{and}\quad  \mathcal{A}_{G_0}\setminus \{p\mid R(G_0;p)=0\}\subseteq \mathcal{D}_{G_0}\subseteq \mathcal{A}_{G_0}.
\end{align*}
\end{theorem}
Theorem~\ref{thm:main equal} can be viewed as a (theoretical) tool for proving results about the reliability zeros. 
One immediate consequence of it is that reliability zeros are not isolated, because the activity-locus is an open set.
We next state some other consequences.
First of all note that for the edge $K_2$ seen as a two-terminal graph we have $\hat{y}_{K_2}(p)=1/p$ and $\mathcal{E}({K_2})=\emptyset$.
 Therefore, the closure of the activity locus of $\mathcal{H}_{K_2}$ is equal to the closed unit disk. 
 Thus Theorem~\ref{thm:main equal} implies the following result of Brown and Colbourn~\cite{brown-1992}.
\begin{corollary}
Reliability roots of series-parallel graphs are dense in the unit disk.
\end{corollary}

To illustrate the usefulness of Theorem~\ref{thm:main equal} as a tool for proving (new) statements about reliability zeros we next state two results that we prove in Section~\ref{sec:two prop}. 
Let us denote the collection of all reliability zeros as
\begin{equation}
\mathcal{Z} :=\{p\in \C\mid R(G;p)=0 \text{ for some connected graph $G$}\}.   
\end{equation}

Our first result gives a criterion for density of reliability zeros in terms of the existence of reliability zeros close to the positive real line.
\begin{proposition}\label{prop:unbounded zeros?}
The closure of the set of reliability zeros, $\overline{\mathcal{Z}}$, is equal to $\mathbb{C}$ if and only if $\mathcal{Z}$ is unbounded if and only if there exists $p>1$ such that $p\in \overline{\mathcal{Z}}$.
\end{proposition}

Our next result says that there are reliability zeros outside of the closed unit disk in essentially each possible direction.
\begin{proposition}\label{prop:zeros outside disk}
Let $p\in \mathbb{C}$ such that $|p|=1$ and $p^k\neq 1$ for $k=1,\ldots,4$. 
Then there exists $\varepsilon=\varepsilon_p>0$ such that the disk $B(p,\varepsilon)$ is contained in the closure of $\mathcal{Z}$.
\end{proposition}
If one could get rid of the constraint that $p^k\neq 1$ for $k=1,\ldots, 4$ in the proposition above, this would imply by Proposition~\ref{prop:unbounded zeros?} that the collection of reliability zeros is dense in the complex plane.

\subsection*{Approximating the reliability polynomial}
To state our results about the complexity of approximately computing the reliability polynomial, we need to formally introduce some computational problems. 

We consider two types of approximation problems, one for the norm of $R(G;p)$ and one for its argument, for each algebraic number $p$ separately. We consider the argument of a complex number to be an element of $\R/2\pi\mathbb{Z}$, and for $\xi\in\R/2\pi\mathbb{Z}$ we take $\text{abs}{(\xi)}=\min_{\theta\in\xi+2\pi\mathbb{Z}}|{\theta}|$.

Let $p$ be a complex number and $r>0$. We call a number $q\in\Q$ an \emph{$r$-abs-approximation of $p$} if $p\neq 0$ implies $e^{-r}\leq \frac{r}{|p|}\leq e^r$. We call a number $\xi\in\Q$ an \emph{$r$-arg-approximation of $p$} if $p\neq 0$ implies
that $\text{abs}{(\xi-\arg(p))}\leq r$. Note that in both cases an approximation of $0$ could be anything.

We then define two approximation problems for each algebraic $p\in\C$:
\vspace{10pt}\\
\begin{tabular}{rl}
    Name: & \textsc{Approx-Abs-Planar-Rel$(p)$} \\
    Input: & A planar graph $H$.\\
    Output: & A $0.25$-abs-approximation of $R(H;p)$.
\end{tabular}
\vspace{10pt} \\
\begin{tabular}{rl}
    Name: & \textsc{Approx-Arg-Planar-Rel$(p)$} \\
    Input: & A planar graph $H$.\\
    Output: & A $0.25$-arg-approximation of $R(H;p)$.
\end{tabular}
\vspace{10pt}\\
We define  the problems \textsc{Approx-Abs-Rel$(p)$} and \textsc{Approx-Arg-Rel$(p)$} in a similar way, putting no restrictions on the input graph $H$. Note that if $p$ is a real number finding a $0.25$-arg approximation of $R(H;p)$ is equivalent to finding its sign (assuming $R(H;p)\neq 0$).

\begin{theorem}\label{thm:main hard}
Let $G_0$ be a two-terminal graph.
Then for any algebraic number $p\in \mathcal{D}_{G_0}\cup \mathcal{D}^{\mathbb{R}}_{G_0}$ such that $R(G_0;p)S(G_0;p)\neq 0$,
\begin{itemize}
    \item  the problems \textsc{Approx-Arg-Rel$(p)$} and \textsc{Approx-Arg-Rel$(p)$} are \textsc{\#P}-hard;
    \item if additionally $G_0$ is planar with its two terminals on the same face, then the problems \textsc{Approx-Abs-Planar-Rel$(p)$} and \textsc{Approx-Arg-Planar-Rel$(p)$} are \textsc{\#P}-hard.
\end{itemize}
\end{theorem}

This theorem has the following concrete corollary.
\begin{corollary}\label{cor:disk is hard}
For each algebraic number $p\in \mathbb{D}\setminus [0,1)$ both the problems \textsc{Approx-Abs-Planar-Rel$(p)$} and \textsc{Approx-Arg-Planar-Rel$(p)$} are \textsc{\#P}-hard.
\end{corollary}
\begin{proof}
Since $\hat{y}_{K_2}(p)=1/p$, it follows that $\mathcal{A}_{K_2}\cup \mathcal{A}_{K_2}^{\mathbb{R}}=\mathbb{D}\setminus[0,1)$.
Since $R(K_2;p)S(K_2;p)=p(1-p)\neq 0$ for any $p\in \mathbb{D}\setminus\{0\}$, it follows by Theorem~\ref{thm:main equal} that $\mathcal{A}_{K_2}=\mathcal{D}_{K_2}$. 
Our proof of Theorem~\ref{thm:main equal} also implies that $\mathcal{A}^{\mathbb{R}}_{K_2}\setminus \{p\in \mathbb{R}\mid R(G_0;p)\neq 0\}\subseteq \mathcal{D}^{\mathbb{R}}_{K_2}\subseteq \mathcal{A}^{\mathbb{R}}_{K_2}$
(See Proposition~\ref{prop:active vs density} below.) Therefore proving that $\mathcal{D}_{K_2}\cup\mathcal{D}^{\mathbb{R}}_{K_2}=\mathbb{D}\setminus [0,1)$.
Since $K_2$ is planar and has its two terminals on the same face, the result follows from Theorem~\ref{thm:main hard}.
\end{proof}
This corollary should be compared with the fact that for $p\in [0,1)$ approximating $R(G;p)$ is easy for all graphs in the sense that there exists a randomized algorithm that on input of a graph $G$ and $\varepsilon>0$ approximates $R(G;p)$ within a multiplicative factor $\exp(\varepsilon)$ in time polynomial in $n/\varepsilon$ due to Karger~\cite{KargerFPRASallterminal}. See also~\cite{ReliseasyGuoJerrum}. 
Corollary~\ref{cor:disk is hard} thus indicates a clear distinction between the complexity of approximating $R(G;p)$ for positive and non-positive values of $p$ inside the unit disk $\mathbb{D}$.

\subsection*{Organization and approach}
Our proofs are based on the framework developed in~\cites{main-roots,main-approx} which in turn take inspiration from~\cites{galanis-2022,chaoticratios}.

In~\cites{main-roots,main-approx} several results concerning zeros and hardness are proved for the chromatic polynomial and more generally the partition function of the random cluster model, both of which are evaluations of the Tutte polynomial.
While the reliability polynomial is also an evaluation of the Tutte polynomial, the framework developed in~\cites{main-roots,main-approx} does not directly apply to it.
However many of the key ideas do. 
The main effort to prove our main theorems is to use these key ideas in the context of the reliability polynomial.

In the next section we gather some preliminaries and state some basic operations and graph constructions.
In Section~\ref{sec:proof main} we provide a proof of Theorem~\ref{thm:main equal}, which we split into several parts. In Section~\ref{sec:two prop} we prove Propositions~\ref{prop:zeros outside disk} and~\ref{prop:unbounded zeros?} and in Section~\ref{sec:density implies hardness} we prove Theorem~\ref{thm:main hard}.

\section{Preliminaries}\label{sec:prel}
Here we collect terminology, notation and preliminaries that will be used frequently in the remainder of the paper. 
Much of what we include here is well known, see e.g.~\cites{brown-1992,royle-2004,brown-2017}.
We give proofs for the sake of completeness and occasionally to be able to build on these proofs for certain specific properties that we need.
\subsection{Graph operations and reliability}\label{sec:graph operations}




The reliability and split reliability polynomials of series-parallel compositions of two-terminal graphs can be computed recursively using the following identities:
\begin{lemma}\label{lem:recursion}
    Let $G_1,G_2$ be two-terminal graphs. We have the following identities on polynomials:
    \begin{align*}
        \SP(G_1\parallel G_2;p)&=\SP(G_1;p)\SP(G_2;p),\\
        R(G_1\series G_2;p)&=R(G_1;p)R(G_2;p),\\
        R(G_1\parallel G_2;p)&=R(G_1;p)\SP(G_2;p)\!+\!\!R(G_2;p)\SP(G_1;p)\!+\!\!R(G_1;p)R(G_2;p),\\
        \SP(G_1\series G_2; p)&=R(G_1;p)\SP(G_2;p)+R(G_2;p)\SP(G_1;p).
    \end{align*}
\end{lemma}
\begin{proof}
    Let $p\in[0,1]$; by definition, this makes $R(G;p)$ and $\SP(G;p)$ the probabilities that $G$ remains connected/becomes an $s-t$ split respectively. We omit the variable $p$ for ease of notation. Then we have the following.
    \begin{itemize}
        \item The graph $G_1\parallel G_2$ is an $s-t$ split if and only if both $G_1$ and $G_2$ are $s-t$ splits; that is 
        $$\SP(G_1\parallel G_2)=\SP(G_1)\SP(G_2).$$
        \item Similarly, the graph $G_1\series G_2$ is connected if and only if both $G_1$ and $G_2$ are connected; that is 
        $$R(G_1\series G_2)=R(G_1)R(G_2).$$
        \item The graph $G_1\series G_2$ is an $s-t$ split if and only if exactly one of $G_1$ and $G_2$ is an $s-t$ split, and the other is connected; that is 
        $$\SP(G_1\series G_2)=R(G_1)\SP(G_2)+R(G_2)\SP(G_1).$$
        \item Finally, the graph $G_1\parallel G_2$ is connected if one between $G_1$ and $G_2$ is connected, while the other is either connected or an $s-t$ split. By inclusion-exclusion on those two conditions we get
        $$R(G_1\parallel G_2)=R(G_1)\SP(G_2)+R(G_2)\SP(G_1)+R(G_1)R(G_2).$$
    \end{itemize}
    Since the polynomial identities above hold for $p\in[0,1]$ the polynomials coincide.
\end{proof}

For a graph $G$ and an edge $e$ of $G$ we denote by $G\setminus e$ (resp. $G/e$)
 the graph obtained from $G$ by deleting (resp. contracting ) the edge $e$.
 As is well known the reliability and split reliability polynomial satisfy a deletion-contraction recurrence.
\begin{lemma}\label{lem:con-del}
    Let $G=(V,E)$ be a graph and let $e\in E$. Then
    $$R(G;p)=pR(G\setminus e;p)+(1-p)R(G/e;p).$$
    Additionally, let $G$ be two-terminal such that $e\neq \{s,t\}$. Then
    $$\SP(G;p)=p\SP(G\setminus e;p)+(1-p)\SP(G/e;p).$$
\end{lemma}
\begin{proof}
    Let $G=(V,E)$ be a graph, and let $e\in E$. There are bijections between:
    \begin{itemize}
        \item connected edge subgraphs of $G$ that contain $e$ and connected edge subgraphs of $G/e$, and
        \item connected edge subgraphs of $G$ that do not contain $e$ and connected edge subgraphs of $G\setminus e$.
    \end{itemize}
    Then
    \begin{align*}
        R(G;p)&=\sum_{\substack{A\subseteq E\\(V,A)\text{ connected}}}(1-p)^{|A|}p^{E\setminus A}\\
        &=(1-p)\sum_{\substack{A\subseteq E\smallsetminus\{e\}\\(V,A\cup\{e\})\text{ connected}}}(1-p)^{|A|}p^{|E\setminus A|-1}
        +p\sum_{\substack{A\subseteq E\smallsetminus\{e\}\\(V,A)\text{ connected}}}(1-p)^{|A|}p^{{|E\setminus A|-1}}\\
        &=(1-p)R(G/e;p)+pR(G\setminus e;p).
    \end{align*}
    The same bijections hold for $s-t$ splits, and as such the same relation holds for the split reliability polynomial. 
\end{proof}
\subsection{Edge interactions}
We denote the Riemann sphere $\C\cup \{\infty\}$ by $\Cext$.
For $p\in \C\setminus\{1\}$ we define the following M\"obius transformation
$$f_p(z)=1+\frac{1-p}{z-1}$$
and observe that it is an involution, that is, for all $z\in\Cext\ f_p(f_p(z))=z$.
Moreover observe that for a two-terminal graph $G$ we have
\begin{equation}\label{eq:virtual is mobius effective}
 f_p(y_G(p))=\hat{y}_G(p),
   \end{equation}
that is, the virtual edge interaction is $f_p$ applied to the effective edge interaction.

From the properties of the reliability and split reliability polynomials for series-parallel composition we obtain the following properties for the edge interactions.
\begin{lemma}\label{lem:ei-comp}
    Let $G_1, G_2$ be two two-terminal graphs and let $p\in\Cext$. Then the following identities hold as rational functions:
    $$y_{G_1\series G_2}=y_{G_1}+y_{G_2}-1;$$
    $$\hat{y}_{G_1\parallel G_2}=\hat{y}_{G_1}\cdot \hat{y}_{G_2}.$$
    Additionally, for any fixed $p_0\in\Cext$ and for any two-terminal graphs $G_1,G_2$ such that \\$\{y_{G_1}(p_0),y_{G_2}(p_0)\}\neq\{\infty\}$ we have
    $$y_{G_1\series G_2}(p_0)=y_{G_1}(p_0)+y_{G_2}(p_0)-1$$
    and for any two-terminal graphs $G_1,G_2$ such that $\{\hat{y}_{G_1}(p_0),\hat{y}_{G_2}(p_0)\}\neq\{0,\infty\}$ we have
    $$\hat{y}_{G_1\parallel G_2}(p_0)=\hat{y}_{G_1}(p_0)\cdot \hat{y}_{G_2}(p_0).$$
\end{lemma}
\begin{proof}
    We use the recursive identities proven in Lemma~\ref{lem:recursion}, and omit the variable $p$ when possible.

    Let $G_1, G_2$ be two two-terminal graphs; observe that
    \begin{align*}
        y_{G_1\series G_2}&=\frac{(1-p)\SP(G_1\series G_2)}{R(G_1\series G_2)}+1\\
        &=\frac{(1-p)\left(R(G_1)\SP(G_2)+R(G_2)\SP(G_1)\right)}{R(G_1)R(G_2)}+1\\
        &=\frac{(1-p)R(G_2)\SP(G_1)}{R(G_1)R(G_2)}+\frac{(1-p)R(G_1)\SP(G_2)}{R(G_1)R(G_2)}+1\\
        &=\left(\frac{(1-p)\SP(G_1)}{R(G_1)}+1\right)+\left(\frac{(1-p)\SP(G_2)}{R(G_2)}+1\right)-1\\
        &=y_{G_1}+y_{G_2}-1
    \end{align*}
    and
    \begin{align*}
        \hat{y}_{G_1\parallel G_2}&=\frac{R(G_1\parallel G_2)}{\SP(G_1\parallel G_2)}+1\\
        &=\frac{\SP(G_1\parallel G_2)+R(G_1\parallel G_2)}{\SP(G_1\parallel G_2)}\\
        &=\frac{\SP(G_1)\SP(G_2)+R(G_1)\SP(G_2)+R(G_2)\SP(G_1)+R(G_1)R(G_2)}{\SP(G_1)\SP(G_2)}\\
        &=\frac{(\SP(G_1)+R(G_1))(\SP(G_2)+R(G_2))}{\SP(G_1)\SP(G_2)}\\
        &=\hat{y}_{G_1}\hat{y}_{G_2}.
    \end{align*}
    So the identities hold as rational functions; that is, on all points where the ratios are well-defined.
\end{proof}
\begin{lemma}\label{lem:multi-series}
    Let $G$ be a two-terminal graph, and denote by $G^{\series_n}$ the series composition of $n$ copies of $G$. Then
    \[\hat{y}_{G^{\series_n}}(p)=\frac{1}{n}\hat{y}_G(p)+\frac{n-1}{n}.\]
\end{lemma}
\begin{proof} 
    Observe that by iterating Lemma~\ref{lem:ei-comp} we obtain $y_{G^{\series_n}}(p)=ny_G(p)$ $-(n-1)$. 
    Then we have
    $$\hat{y}_{G^{\series_n}}(p)=\frac{1-p}{n(y_G(p)-1)}+1=\frac{1-p}{n\frac{1-p}{\hat{y}_G(p)-1}}+1=\frac{1}{n}\hat{y}_G(p)+\frac{n-1}{n}.$$
\end{proof}
\subsection{Gadgets}\label{sec:gadget}
Take a two-terminal graph $(G,s,t)$, which we refer to as a \emph{gadget}, and another graph $H$. We can substitute the gadget into an edge $e$ of $H$ by removing $e$, adding a copy of $G$, and identifying the source and sink of $G$ with the endpoints of $e$; we denote the resulting graph as $H(G)_e$. See Figure~\ref{fig:gadget} below for an example. While technically it matters which endpoint of $e$ is identified with $s$, for the purpose of the reliability polynomial it does not matter as we show below and therefore we don't specify the choice of identification.

 \begin{center}
    \vspace{12pt}
\begin{figure}[ht]
\begin{subfigure}{0.29\textwidth}
\centering
    \begin{tikzpicture}[scale=0.5]
    \node[circle,fill,inner sep=2pt,label=above:$s$] (s) at (0,6) {};
    \node[circle,fill,inner sep=2pt,label=below:$t$] (t) at (0,0) {};
    \draw[thick] (s) edge [bend left=45] (t);
    \draw[thick] (s) edge [bend left=30] (t);
    \draw[thick] (s) edge [bend left=15] (t);
    \draw[thick] (s) edge [bend right=45] (t);
    \draw[thick] (s) edge [bend right=30] (t);
    \draw[thick] (s) edge [bend right=15] (t);
\end{tikzpicture}
    \caption{$G$}
\vspace{10pt}\end{subfigure}
\hfill
\begin{subfigure}{0.29\textwidth}
\centering
\begin{tikzpicture}[scale=0.5]
    \node[circle,fill,inner sep=2pt] (s) at (2,6) {};
    \node[circle,fill,inner sep=2pt] (0) at (0,3) {};
    \node[circle,fill,inner sep=2pt] (1) at (4,3) {};
    \node[circle,fill,inner sep=2pt,label=below:\phantom{t}] (t) at (2,0) {};
    \draw[thick] (s) -- (0) -- (t);
    \draw[thick] (s) -- (1) -- (t);
    \draw[thick] (0) edge["$e$"] (1);
\end{tikzpicture}
    \caption{$H$}
\vspace{10pt}\end{subfigure}
\hfill
\begin{subfigure}{0.39\textwidth}
\centering
\begin{tikzpicture}[scale=0.5]
    \node[circle,fill,inner sep=2pt] (s) at (2,6) {};
    \node[circle,fill,inner sep=2pt] (0) at (0,3) {};
    \node[circle,fill,inner sep=2pt] (1) at (4,3) {};
    \node[circle,fill,inner sep=2pt,label=below:\phantom{t}] (t) at (2,0) {};
    \draw[thick] (s) -- (0) -- (t);
    \draw[thick] (s) -- (1) -- (t);
    \draw[thick] (0) edge [bend left=45] (1);
    \draw[thick] (0) edge [bend left=30] (1);
    \draw[thick] (0) edge [bend left=15] (1);
    \draw[thick] (0) edge [bend right=45] (1);
    \draw[thick] (0) edge [bend right=30] (1);
    \draw[thick] (0) edge [bend right=15] (1);
\end{tikzpicture}
    \caption{$H(G)_e$}
\vspace{10pt}\end{subfigure}
        
\caption{An example of gadget substitution.}
\label{fig:gadget}
\end{figure}
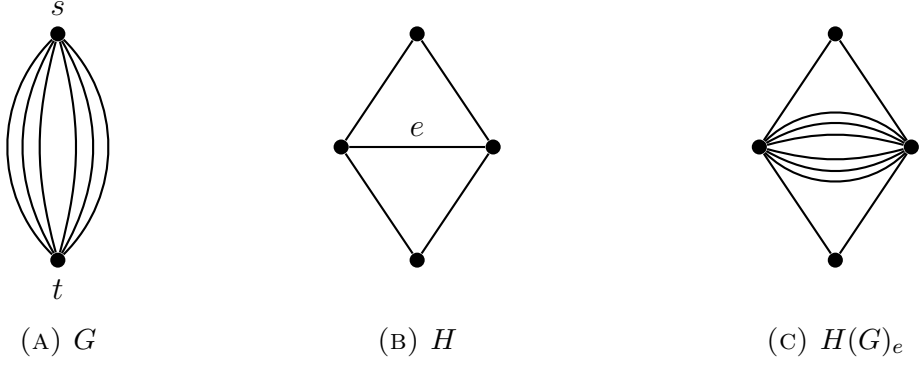\vspace{-30pt}
\end{center}

\begin{lemma}\label{lem:single-replacement}
    Let $H=(V,E)$ be a graph, $G$ be a gadget, $e$ be an edge of $H$ that is not a loop. 
    Then
    $$R(H(G)_e;p)=\SP(G;p)R(H\setminus e;p)+R(G;p)R(H/e;p).$$
    Now let $H$ be two-terminal. Then
    $$\SP(H(G)_e;p)=\SP(G;p)\SP(H\setminus e;p)+R(G;p)\SP(H/e;p).$$
\end{lemma}
\begin{proof}
    Let $H=(V(H),E(H))$ be a graph and $G=(V(G),E(G))$ a gadget; let $H(G)_e=(V,E)$. Consider the structure of a connected edge subgraph of $H(G)_e$; every vertex of $G$ has to be connected to at least one of $s,t$ to be connected to the vertices of $H$.
    A subset $A\subseteq E$ naturally decomposes into $E_H\cup E_G$, with $A_H\subseteq E(H)$, $A_G\subseteq E(G)$ and $E_H\cap E_G=\emptyset$.
    Then a spanning subgraph $(V,A)$ of $H(G)_e$ is connected if and only if one of the following happens:
    \begin{itemize}
        \item $(V(G),A_G)$ is a connected subgraph of $G$ and $A_{H}$ induces a connected subgraph of $H/e$, which are in bijection with the connected edge subgraphs of $H$ which contain $e$ (see the proof of Lemma~\ref{lem:con-del}),
        \item $(V(G),A_{G})$ is a split subgraph of $G$ and $A_{H}$ induces a connected subgraph of $H\setminus e$ which are in bijection with the connected subgraphs of $H$ which do not contain $e$.
    \end{itemize}
    Then for $p\in [0,1]$ the probability of $H(G)_e$ being connected is thuis given by
    $$R(H(G)_e;p)=\SP(G;p)R(H\setminus e;p)+R(G;p)R(H/e;p)$$
    and since those are polynomial identities, they hold for all $p\in\C$.

    The same bijections hold for $s-t$ splits, and so the same relation holds for the split reliability polynomial.
\end{proof}


The next lemma combines the previous lemma with the deletion contraction recurrences.

\begin{lemma}\label{lem:eei}
Let $H$ be a graph, $e$ an edge of $H$, and let $G$ be a two-terminal graph and $p\in \C$ such that $R(G;p)\neq 0$. Then
    $$\frac{1-p}{R(G;p)}R(H(G)_e;p)=R(H;p)+\left(y_G(p)-(p+1)\right)R(H\setminus e;p).$$
\end{lemma}
\begin{proof}
    We observe that for any graph with more than one vertex (specifically, for any two-terminal graph) we have $R(G;1)=0$, so we may assume $p\neq 1$.
    We use the contraction-deletion identity (Lemma~\ref{lem:con-del}) for the reliability polynomial to arrive at
    $$R(H/e;p)=\frac{R(H;p)-pR(H\setminus e;p)}{1-p}.$$
    We then substitute this into Lemma~\ref{lem:single-replacement} to obtain
    \begin{align*}
        \frac{1-p}{R(G;p)}R(H(G)_e;p)&=\frac{1-p}{R(G;p)}\left(\SP(G;p)R(H\setminus e;p)+R(G;p)R(H/e;p)\right)\\
        &=(1-p)\frac{\SP(G;p)}{R(G;p)}R(H\setminus e;p)+R(H;p)-pR(H\setminus e;p)\\
        &=R(H;p)+(y_G(p)-(p+1))R(H\setminus e;p),
    \end{align*}
    as desired.
\end{proof}

We can also replace every edge of a graph $H$ with a two-terminal graph $G$.
\begin{lemma}\label{lem:substitution}
    Let $H,G$ be two-terminal graphs, with $G$ connected and let $H(G)$ be the graph obtained by substituting $G$ into every edge of $H$. Then the virtual edge interaction of $H(G)$ does not depend on how we orient the edges of $H$ and satisfies
    \begin{equation}
     \hat{y}_{H(G)}(p)=\hat{y}_H\left(\frac{1}{\hat{y}_G(p)}\right).\label{eq:substitution}  
    \end{equation}
\end{lemma}

\begin{proof}
    We will show that
    \begin{align}
        R(H(G);p)&=\left(R(G;p)+\SP(G;p)\right)^{|E(H)|}R\left(H;\frac{1}{\hat{y}_G(p)}\right)\ \ \label{eq:edge replacement interaction} \text{and}\\
        \SP(H(G);p)&=\left(R(G;p)+\SP(G;p)\right)^{|E(H)|}\SP\left(H;\frac{1}{\hat{y}_G(p)}\right), \nonumber 
    \end{align}
    from which the statement of the lemma directly follows.
    Observe that $R(G;p)$ is zero as a polynomialproof if and only if $G$ is disconnected. We can then assume that $p\in [0,1]$ is such that $R(G;p)\neq 0$. We work by induction on the number of edges of $H$, starting with the case $|E(H)|=1$. We may assume $H=K_2$ (for otherwise the statement is trivial). 
    Since $K_2(G)=G$, it is not difficult to see that the desired statement holds.
    Let us next assume that $|E(H)|>1.$ 
    Let us view $H(G)$ as $H'(G)_e$ for some edge $e$ of $H$. (Here $H'$ is the graph obtained from $H(G)$ by replacing the copy of $G$ on $e$ by $K_2$.)
    By Lemma~\ref{lem:single-replacement} we then have
    \begin{align*}
        R(H(G)&;p)=\,R(G;p)R((H/e)(G);p)+\SP(G;p)R((H\setminus e)(G);p)\\
        =&\,\left(R(G;p)+\SP(G;p)\right)\!\!\Bigg(\!\!\!\left(1-\frac{1}{\hat{y}_G(p)}\right)R((H/e)(G);p)
        +\frac{1}{\hat{y}_G(p)}R((H\setminus e)(G);p)\!\!\Bigg).
    \end{align*}
Since $H/e,H\setminus e$ both have fewer edges than $H$, the statement follows by induction and the deletion contraction recurrence.
    The same steps work if we substitute $R(H(G);p)$ with $\SP(H(G);p)$. 
\end{proof}
\begin{remark}
Note that Equation~\eqref{eq:edge replacement interaction} partly motivates the terminology of effective and virtual interactions. The reliability polynomial of $H(G)$ at $p$ is up to a factor equal to the reliability polynomial of $H$ evaluated at $1/\hat{y}_G(p)$.    
\end{remark}

\subsection{Multivariate reliability and split reliability}
In this section we will define the multivariate analogue of the reliability and split reliability polynomial of a graph. As we will see this will help us to derive a formula for the the general edge replacement, where we replace each edge $e_i$ of $H$ with different two-terminal graphs $G_i$.
For a graph $G=(V,E)$ with edges $\{e_1,\dots,e_m\}$ let $p_1,\dots,p_m$ variables. Then the multivariate reliability polynomial is defined as
\[
 R(G;{\bf p})=\sum_{\substack{A\subseteq E\\ (V,A) \textrm{ connected}}} \prod_{e_i\in A} (1-p_i) \prod_{e_i\in E\setminus A} p_i
\]
and if $G$ is a two-terminal graph, then the multivariate split reliability of $G$ is defined as
\[
 S(G;{\bf p})=\sum_{\substack{A\subseteq E\\ (V,A) \textrm{ $s$--$t$ split}}} \prod_{e_i\in A} (1-p_i) \prod_{e_i\in E\setminus A} p_i
\]
Analogously to the definitions in the introduction let us define the  multivariate virtual effective edge interaction of $G$ at $\bf p$ as
\[
\hat{y}_{G}({\bf p})=\frac{R(G;{\bf p})}{S(G;{\bf p})}+1.
\]
\begin{lemma}\label{lemma:multivariate edge replacement}
Let $H$ be a two-terminal graph with edges $\{e_1,\dots,e_m\}$ and let $G_1,\dots,G_m$ be two-terminal graphs. Let $H(G_1,\dots,G_m)$ be a graph obtained from $H$ by replacing, for $i=1,\ldots,m$, the edge $e_i$ with the two-terminal graph $G_i$ (i.e. we delete the edge $e_i=(u,v)$ and we identify the source (resp. target) of $G_i$ with $u$ (resp. $v$). 
Then
\[
\hat y_{H(G_1,\dots, G_m)}(p)=\hat y_{H}({\bf p})\Big|_{p_i=\tfrac{1}{\hat y_{G_i}(p)}}
\]
as rational functions.
\end{lemma}
We omit the proof as it follows in exactly the same way as Lemma~\ref{lem:single-replacement}.

\subsection{Normal families and the Montel-Carathèodory Theorem}
We recall here the definition of a normal family and state the Montel-Carathèodory Theorem. Further background and proofs of these results can be found in Chapter 3.2 of \cite{book-normal}.
\begin{definition}
    A family $\mathcal{F}$ of meromorphic functions on an open set $U\subseteq\Cext$ is called \emph{normal} if every sequence $(f_n)\subseteq\mathcal{F}$ contains a subsequence which converges spherically uniformly on compact subsets of $U$. 
\end{definition}
\begin{lemma}
    Let $(f_n)$ be a sequence of meromorphic functions on an open set $U\subseteq\Cext$ which converges spherically uniformly on compact subsets to $f$. Then $f$ is either a meromorphic function on $U$ or identically equal to $\infty$.
\end{lemma}
\begin{theorem}[Montel-Carathèodory]\label{thm:montel}
    Let $\mathcal{F}$ be a family of meromorphic functions on an open set $U\subseteq\Cext$. If there exist three distinct points $\{a,b,c\}\subseteq\Cext$ such that $\forall f\in\mathcal{F}$ and $\forall u\in U$ we have $f(u)\notin\{a,b,c\}$ then $\mathcal{F}$ is normal. 
\end{theorem}

\section{Activity, density and zeros}\label{sec:proof main}
In this section we prove Theorem~\ref{thm:main equal}.
We start with the relation between the activity-locus and the zero-locus.

\subsection{Relating the activity-locus and the zero-locus}\label{sec:active vs zeros}
We begin by stating a lemma on the values of the virtual edge interactions (which we recall are defined in~\eqref{eq:def virtual}) at reliability zeros.
\begin{lemma}\label{lem:3point}
    Let $p\in \C$ and let $G_0$ be a two-terminal graph such that $\SP(G_0;p)\neq -R(G_0;p)$. 
    Then the following are equivalent:
    \begin{itemize}
        \item[(1)] $R(G;p)=0$ for some $G\in\mathcal{H}_{G_0}$
        \item[(2)] $\hat{y}_G(p)=1$, $\SP(G;p)\neq 0$ 
        for some $G\in \mathcal{H}_{G_0}$
        \item[(3)] $\hat{y}_G(p)\in\{\omega,\omega^2,1\}$, $\SP(G;p)\neq 0$ 
        for some $G\in \mathcal{H}_{G_0}$ and $\omega$ a primitive third root of unity.
    \end{itemize}
\end{lemma}
\begin{proof}
We start with the implication from statement (2) to statement (1).
    Note that $\hat{y}_G(p)=1$ implies
    $\frac{R(G;p)}{\SP(G;p)}=0$, from which it follows that $R(G;p)=0$. 

 To see the reverse implication let $G\in \mathcal{H}_{G_0}$ be a edge minimal such that $R(G;p)=0$.
 Now $\hat{y}_{G}(p)=\frac{0}{\SP(G;p)}+1$. 
 If $\SP(G;p)\neq 0$, this is well defined and equal to $1$. 
 So let us assume instead $\SP(G;p)=0$. 
 By minimality of $G$ we may assume by Lemma~\ref{lem:recursion} that $G$ is not the series composition of two smaller graphs in $\mathcal{H}_{G_0}$.
 Since by assumption $\SP(G_0;p)\neq -R(G_0;p)$ we have $G\neq G_0$ and thus $G$ is the parallel composition of two smaller graphs $G_1,G_2\in \Hs_{G_0}$.
 Since by Lemma~\ref{lem:recursion} $S(G_1\parallel G_2;p)=S(G_1;p)S(G_2;p)$, we may assume that $S(G_1;p)=0$ and moreover we may assume that $G_2$ itself is not the parallel composition of two smaller graphs in $\mathcal{H}_{G_0}$.
 By Lemma~\ref{lem:recursion} we have 
 \begin{align*}
 0=R(G;p)&=R(G_1\parallel G_2;p)=R(G_1;p)S(G_2;p)+S(G_1;p)R(G_2;p)+R(G_1;p)R(G_2;p)
 \\
 &=R(G_1;p)(R(G_2;p)+S(G_2;p)).
 \end{align*}
 and hence $R(G_2;p)+S(G_2;p)=0$.
By assumption $G_2$ is thus not equal to $G_0$ and must therefore by the series composition of two smaller graphs $H_1,H_2\in \mathcal{H}_{G_0}$.
Now consider the graph $\hat{G_2}$ obtained from $G_2$ by identifying its terminal vertices into a single vertex.
It is not difficult to see that $R(\hat{G_2};p)=R(G_2;p)+S(G_2;p)=0$.
Since $G_2=H_1\bowtie H_2$, it follows that as graphs $\hat{G_2}=H_1^T\parallel H_2$ and therefore $R(H_1^T\parallel H_2;p)=0$.
This is however a contradiction to the minimality assumption of $G$. We conclude that statement (2) implies statement (1).

Note that the implication from statement (3) to statement (2) is trivial and that statement (2) implies statement (3) by letting $G'=G\parallel G\parallel G$, from which it follows that $\hat{y}_{G'}(p)=\hat{y}_G(p)^3=1$ and $SP(G';p)=SP(G;p)^3\neq 0$.
This finishes the proof.
\end{proof}

The next proposition immediately implies the first equality of Theorem~\ref{thm:main equal}.
\begin{proposition}\label{prop:zeros vs active}
Let $G_0$ be a two-terminal graph.
Then \begin{equation}
\overline{\mathcal{Z}_{G_0}}=\overline{\mathcal{A}_{G_0}}.    \end{equation}
\end{proposition}
\begin{proof}
    Let $p\in \mathcal{Z}_{G_0}$. Then by Lemma~\ref{lem:3point} there exists $G\in\Hs_{G_0}$ such that 
    $\hat{y}_G(p)=\frac{R(G;p)}{\SP(G;p)}+1=1$ and $\SP(G;p)\neq 0$.
    We denote $f(p)$ for the map $p\mapsto \hat{y}_G(p)$, with $G$ fixed.
    Since $\SP(G;p)$ and $R(G;p)$ are distinct polynomials in $p$, we can find $\delta>0$ such that $f|_{B(p,\delta)}$ is holomorphic and not constant and so by the open mapping theorem, it is an open map. 
    Let $\epsilon>0$ such that $\epsilon<\delta$. 
    Then $f(B(p,\epsilon))$ is mapped to an open set of $\C$ containing $f(p)=1$, which has to contain some disc $B(1,\rho)$; particularly, it contains the point $1+i\tfrac{\rho}{2}$, which is greater than one in absolute value and not real.
    
    We have shown that for each $\epsilon>0$ sufficiently small ($\epsilon<\delta$) we can find $q_{\epsilon}\in B(p,\epsilon)$ such that $1<|{f(q_{\epsilon})}|<\infty$ and such that $f(q_{\epsilon})\notin \mathbb{R}$. We can thus build a sequence $(p_n)$ with $p_n:=q_{1/n}$ for $n\in\mathbb{Z}_{>0}$, which converges to $p$ such that $p_n\in\mathcal{A}_{{G_0}}$ for $n$ large enough, implying that $p\in \overline{ \mathcal{A}_{G_0}}$. This shows that $\overline{\mathcal{Z}_{G_0}}\subseteq \overline{\mathcal{A}_{G_0}}$.

We next focus on the reverse inclusion.
    Let $p\in \mathcal{A}_{{G_0}}$, and let $G\in\Hs_{G_0}$ such that $1<|{\hat{y}_G(p)}|<\infty$ and such that $\hat{y}_G(p)\notin \mathbb{R}$.  
    We intend to prove that for all open neighbourhoods $U$ of $p$ we have $U\cap\mathcal{Z}_{G_0}\neq\emptyset$, that is $p\in\overline{\mathcal{Z}_{G_0}}$.
Assume to the contrary that $U\cap\mathcal{Z}_{G_0}=\emptyset$ for some open set $U$ containing $p$, which we may assume to be disjoint from $\mathcal{E}(G_0)$.
By Lemma~\ref{lem:3point} we have that for all $p'\in U$ $\hat{y}_G(p')\notin\{\omega,\omega^2,1\}$ for all $G\in \Hs_{G_0}$.
By the Montel-Carathèodory theorem (Theorem~\ref{thm:montel}) the family $\mathcal{F}\!=\!\{p\mapsto \hat{y}_G(p)\}_{G\in\Hs_{G_0}}$ of functions on $U$ is normal. 
We will show that this cannot be the case by exhibiting a sequence $(f_n)$ of elements in $\mathcal{F}$ which converges to a discontinuous function (and so cannot have a subsequence which converges to a meromorphic function on compact subsets of $U$ and hence cannot be normal).
This contradiction then shows that $U\cap\mathcal{Z}_{G_0}\neq \emptyset$ for each open set $U$ containing $p$. 

We will now construct the desired sequence.
Denote by $H_-$ the open half-plane $\{z\in\C\mid\Re(z)<1\}$ and by $H_+$ the open half-plane $\{z\in\C\mid\Re(z)>1\}$.
Let $f$ be the rational function defined by $p\mapsto \hat{y}_G(p).$
Let $U$ be a small enough neighbourhood of $p$ such that $f(U)\cap \overline{\mathbb{D}}=\emptyset$. 
This exists since $|f(p)|>1$.
Since the map $f(p)$ is not constant it follows that $f(U)$ is an open set of $\widehat{\C}$. In particular, $\arg(f(U))$ contains an interval of non-zero Lebesgue measure. 

Now, consider the sequence $G_n=G^{\parallel_n}$. By Lemma~\ref{lem:ei-comp} we have that $\hat{y}_{G_n}(U)=(\hat{y}_G(U))^n$ (pointwise powers), so $\arg(\hat{y}_{G_n}(U))=n\arg(\hat{y}_G(U))$. We can then choose a big enough $n\in\N$ such that $\arg(\hat{y}_{G_n}(U))$ covers the entire circle.

Consider two points $p_0,p_1\in U$ such that $\hat{y}_{G_n}(p_0)\in H_-$ and $\hat{y}_{G_n}(p_1)\in H_+$ (which exist, since $\arg(\hat{y}_{G_n}(U))$ covers the entire circle and $|(\hat{y}_{G_n}(U))|>1$ everywhere). By Lemma~\ref{lem:multi-series} we have that $\hat{y}_{G^{\series_m}}(p)=\frac{1}{m}\hat{y}_G(p)+\frac{m-1}{m}$. We can interpret this result as a convex combination of $\hat{y}_G(p)$ and $1$, which approaches $1$ as $m$ grows. Then for big enough $m$ we have that 
\begin{align*}
    \hat{y}_{(G^{\parallel_n})^{\series_m}}(p_0)=\frac{\hat{y}_G(p_0)^n}{m}+\frac{m-1}{m}\in \mathbb{D}\\
    \hat{y}_{(G^{\parallel_n})^{\series_m}}(p_1)=\frac{\hat{y}_G(p_1)^n}{m}+\frac{m-1}{m}\in \C\setminus\overline{\mathbb{D}}
\end{align*}
Then the sequence of functions $g_k:U\to\Cext$ defined as 
$$g_k(p)=\hat{y}_{\left((G^{\parallel_n})^{\series_m}\right)^{\parallel_k}}(p)=\left(\frac{\hat{y}_G(p)^n}{m}+\frac{m-1}{m}\right)^k$$
is a sequence of elements of $\mathcal{F}$ that converges to a non-continuous function on $U$, as desired.
\end{proof}

\subsection{Relating the activity-locus and the the density locus}\label{sec:active vs dense}
Let $\varepsilon>0$: we call a set $S\subseteq\C$ \emph{$\varepsilon$-dense} if $\forall z\in\C\ \exists z'\in S$ such that $|z-z'|\!<\!\varepsilon$. 
The following lemma is~\cite{main-approx}*{Lemma 4.2} and will be useful for us.
\begin{lemma}\label{lem:4.2}
    Let $\varepsilon > 0$ and let $a, b, c \in B(0, \varepsilon)$ such that the convex cone spanned by $a, b, c$ is $\C$. Then the set $a\N + b\N + c\N$ is $\varepsilon$-dense in $\C$.
\end{lemma}

For a two terminal graph $G_0$ and $p\in \mathbb{C}$ define the M\"obius transformation
\begin{equation}\label{eq:define mobius}
g(z)=f_p(f_p(z)f_p(\YGz))=\frac{z\YGz-p}{\YGz+z-1-p}.
\end{equation}
Note that by Lemma~\ref{lem:ei-comp} and properties of $f_p$ that $g(y_G)$ is equal to $y_{G\series G_0}$.
Clearly, $g$ has at most two fixed points. The next lemma classifies these in case $p\in \mathcal{A}_{G_0}\cup\mathcal{A}^{\R}_{G_0}$.
\begin{lemma}\label{lem:mobius}
Let $G_0$ be a two-terminal graph and suppose that $p\in \mathcal{A}_{G_0}\cup\mathcal{A}^{\R}_{G_0}$.
Let $g$ be the M\"obius transformation as defined in~\eqref{eq:define mobius}.
Then $g(1)=1$ and $g(p)=p$. Moreover, $1$ is an attracting fixed point of $g$, while $p$ is a repelling fixed point of $g$.
\end{lemma}
\begin{proof}
It follows from a direct calculation that $1$ and $p$ are fixed points of $g$.
The derivative of $g$, $g^\prime(z)$, satisfies
\begin{equation*}
g^\prime(z)=\frac{(\YGz-p)(\YGz-1)}{(\YGz+z-1-p)^2}.    
\end{equation*}
Hence $g'(1)=\frac{\YGz-1}{\YGz-p}$ and $g'(p)=\frac{\YGz-p}{\YGz-1}$. 
Since $\frac{\YGz-1}{\YGz-p}=\frac{1}{\hat{y}_{G_0}(p)}$, we obtain by assumption that $|g'(1)|<1$ and $|g'(p)|>1$. 
In other words that $1$ is an attracting fixed point of $g$ and $p$ is a repelling fixed point of $g$.
\end{proof}

The second part of Theorem~\ref{thm:main equal} follows directly from the next proposition.

\begin{proposition}\label{prop:active vs density}
Let $G_0$ be a two-terminal graph. 
Then 
\begin{align*}
\mathcal{D}_{G_0}\subseteq  \mathcal{A}_{G_0} \quad \text{and}\quad  \mathcal{A}_{G_0}\setminus \{p\mid R(G_0;p)=0\}\subseteq \mathcal{D}_{G_0}, 
\\
\mathcal{D}^{\R}_{G_0}\subseteq  \mathcal{A}^{\R}_{G_0} \quad \text{and}\quad  \mathcal{A}^{\R}_{G_0}\setminus \{p\mid R(G_0;p)=0\}\subseteq \mathcal{D}^\R_{G_0}. 
\end{align*}
\end{proposition}
\begin{proof}
Let $p\in \mathcal{A}_{G_0}\cup\mathcal{A}^{\R}_{G_0} \setminus \{p\mid R(G_0;p)=0\}$.
In what follows we omit the argument $p$ to the (virtual) edge interactions for readability.
We first consider part of the argument that is the same for the real and non-real case after which we distinguish between these two cases.

Consider the sequence of graphs starting from $G_0$ where $G_{n+1}=G_n\parallel G_0$. 
Note that a simple induction argument shows that the associated effective edge interactions $g_n=y_{G_n}$ are given by $g(g_{n-1})$, where $g$ is the M\"obius transformation defined in~\eqref{eq:define mobius}.
By Lemma~\ref{lem:mobius}, the sequence $(g_n)$ either converges to the attracting fixed point $1$ or is constantly equal to the repelling fixed point $p$. 
Since $g_0=\YGz=p$ implies $\hat{y}_{G_0}(p)=f_p(\YGz)=0$, while $|\hat{y}_{G_0}(p)|>1$, so the latter is impossible. 

We next claim that 
\begin{equation}\label{eq:claim R=0}
\text{if $R(G_n;p)=0$ for some $p$, then $S(G_n;p)\neq 0$ and hence $g_n=\infty$.}    
\end{equation} 
In particular this can only happen for a single value of $n$ since the sequence $(g_n)$ converges to $1$.
To prove~\eqref{eq:claim R=0}, suppose $R(G_n;p)=0$ and $S(G_n;p)=0$. 
By Lemma~\ref{lem:recursion} and a simple induction argument we have $S(G_0;p)=0$ and hence $0=R(G_n;p)=R(G_0;p)^{n+1}$, contradicting $S(G_0;p)\neq -R(G_0;p)$.

Since the sequence $(y_{G_n})$ converges to $1$, for all $\epsilon>0$ there exists an index $m_{\epsilon}$ such that $\forall n\geq m_{\epsilon}:\ |y_{G_n}-1|<\epsilon$. 

If $p\in \mathbb{R}$ we have $\hat{y}_{G_0}(p)<-1$ and hence $y_{G_n}-1$ alternates in sign, because $g'(1)=1/\hat{y}_{G_0}(p)<0$. 
Therefore there exists $n\geq m_{\epsilon}$ such that $R(G_n;p)R(G_{n+1};p)\neq 0$, $\hat{y}_{G_n}(p)-1\in (-\varepsilon,0)$ and $\hat{y}_{G_n+1}(p)-1\in (0,\varepsilon)$.
This implies that $\mathbb{N}(y_{G_n}-1)+\mathbb{N}(y_{G_{n+1}}-1)$ is $\varepsilon$-dense in $\mathbb{R}$ and by Lemma~\ref{lem:ei-comp} consists of shifted effective edge interactions of series compositions of $G_n$ and $G_{n+1}$. 
By Lemma~\ref{lem:recursion} $R(G_n^{\series k}\series G_{n+1}^{\series \ell};p)\neq 0$ for all $k,\ell$.
It follows that $\{y_{G}\mid G\in \mathcal{H}_{G_0} ,R(G;p)\neq 0\}$ is $\varepsilon$-dense in $\mathbb{R}$. 
As this holds for all $\varepsilon>0$, it follows that $p\in \mathcal{D}^{\R}_{G_0}$ in case $p$ is real, thus proving the inclusion $\mathcal{A}^{\R}_{G_0}\setminus \{p\mid R(G_0;p)=0\}\subseteq \mathcal{D}^{\R}_{G_0}$.

In case $p\notin \mathbb{R}$ the argument is slightly more involved.
In this case we claim that the values $y_{G_n}-1$ are not contained in a half-plane. 
Indeed, observe that
\[
\arg({y}_{G_{n+1}}-1)-\arg({y}_{G_n}-1)=\arg\left(\frac{y_{G_{n+1}}-1}{y_{G_n}-1}\right)=\arg\left(\frac{g(g_n)-g(1)}{g_n-1}\right).
\]
It follows that
\[
\arg({y}_{G_{n+1}}-1)-\arg({y}_{G_n}-1)\xrightarrow{n\to\infty}\arg(g'(1))=\arg\left(\frac{1}{\hat{y}_{G_0}(p)}\right).
\]
Since $\hat{y}_{G_0}(p)$ is non-real by hypothesis, the consecutive difference between the arguments of the sequence ${y}_{G_n}-1$ converges to $\arg(1/\hat{y}_{G_0}(p))\neq 0\mod 2\pi$. 
Therefore for each $\varepsilon>0$ there exist three indices $n_a,n_b,n_c\geq m_\varepsilon$ such that ${y}_{G_{n_a}}-1,{y}_{G_{n_b}}-1,{y}_{G_{n_c}}-1$ satisfy the hypothesis of Lemma~\ref{lem:4.2}. We may assume that none of these interactions are equal to $\infty$, and thus $R(G_{n_i};p)\neq 0$ for each $i\in \{a,b,c\}$ by the previous claim~\eqref{eq:claim R=0}.
Thus by Lemma~\ref{lem:4.2} the set ${H}_\varepsilon=({y}_{G_{n_a}}-1)\N+({y}_{G_{n_b}}-1)\N+({y}_{G_{n_c}}-1)\N$ is $\varepsilon$-dense in $\C$.
Since ${y}_{G_1\series G_2}-1={y}_{G_1}-1+{y}_{G_2}-1$ by Lemma~\ref{lem:ei-comp}, the set $H_\varepsilon$ consists of (shifted) effective edge interactions of series compositions of the graphs $G_{n_a},G_{n_b},G_{n_c}$, which in turn are parallel compositions of $G_0$. 
By Lemma~\ref{lem:recursion} we have that $R(G;p)\neq 0$ for each of these compositions.
In particular, $ {H}_\varepsilon+1$ and therefore $\{y_{G}\mid G\in \mathcal{H}_{G_0}, R(G;p)\neq 0\}$ is $\varepsilon$-dense in $\C$. Since this holds for every $\varepsilon>0$ we obtain that $\{y_G\mid G\in \mathcal{H}_{G_0}, R(G;p)\neq 0\}$ is dense in $\C$.
It thus follows that $p\in \mathcal{D}_{G_0}$ proving the inclusion $\mathcal{A}_{G_0}\setminus \{p\mid R(G_0;p)=0\}\subseteq \mathcal{D}_{G_0}$.

The other inclusion is easier to proof.
Indeed, suppose $p\in \mathcal{D}_{G_0}$ (resp. $p\in  \mathcal{D}^{\R}_{G_0}$).
Then the set of values $\{y_G(p)\mid G\in \Hs_{G_0}, R(G;p)\neq 0\}$ is dense in $\C$ (resp. dense in $\mathbb{R})$.
Since the map $f_p$ is a M\"obius transformation, it follows that $\{\hat{y}_G(p)\mid G\in \Hs_{G_0},R(G;p)\neq 0\}=\{f_p(y_G(p))\mid G\in \Hs_{G_0},R(G;p)\neq 0 \}$ is dense in $f_p(\C)=\widehat{\C}\setminus \{1\}$  (resp. dense in $f_p(\R)=\widehat{\R}\setminus \{1\}$).
Therefore there exists $G\in \Hs_{G_0}$ such that $1<|\hat{y}_G(p)|<\infty$ and such that $\hat{y}_G(p)\notin \mathbb{R}$ (resp. such that $-\infty<\hat{y}_G(p)<-1$).
This proves the other inclusion for both the real and non-real case.
\end{proof}

\section{Two propositions on the activity-locus}\label{sec:two prop}
\subsection{Boundedness of reliability zeros versus activity near the positive real axis}
Here we provide a proof of Proposition~\ref{prop:unbounded zeros?}.
We start with some results about activity loci.
It will be convenient to define
\[
\mathcal{A}:=\bigcup_{G}\mathcal{A}_G,
\]
where the union is over all two-terminal graphs $G$. 

\begin{lemma}\label{lem:closure of all activity and zeros}
We have the equality
\[
\overline{\mathcal{A}}=\overline{\mathcal{Z}}.
\]
\end{lemma}
\begin{proof}
By Proposition~\ref{prop:zeros vs active} we have the following.
    \[
        \overline{\mathcal{A}}=\overline{\bigcup_{G}\mathcal{A}_G}=\overline{\bigcup_{G}\overline{\mathcal{A}_G}}=\overline{\bigcup_{G}\overline{\mathcal{Z}_G}}=\overline{\bigcup_{G}\mathcal{Z}_G}.
    \]
    It thus suffices to show that $\cup_G \mathcal{Z}_G=\mathcal{Z}$. To prove this, note that clearly, $\cup_G\mathcal{Z}_G\subseteq \mathcal{Z}$. To prove the other direction of the containment, assume that there is a $p\in\mathcal{Z}\setminus \cup_G\mathcal{Z}_G$. This means that for any graph $G$ if $R(G,p)=0$, then $p\in\mathcal{E}(G)$ for any choice of terminals, while there exists a graph $G_0$, such that $R(G_0,p)=0$. Let $G_0$ to be such a graph with minimum number of vertices and edges and let $s,t$ be two distinct terminals for $G_0$. Since $p\in\mathcal{E}(G_0)$, we have
    \[
    0=R(G_0,p)=-S(G_0,p).
    \]
    On the other hand for the graph $\hat{G}_0$ obtained from $G_0$ by identifying its two terminals we have $R(\hat{G}_0;p)=R(G_0;p)+S(G_0;p)=0$, which contradicts to the choice of $G_0$ and finishes the proof.
\end{proof}

\begin{lemma}\label{lem:zero-free}
    Let $U\subseteq \C$ be an open set such that $U\cap\mathcal{A}=\emptyset$.
    Then for any two-terminal graph $G$ and each $p\in U\setminus \mathcal{E}(G)$ it holds that $1/\hat{y}_G(p)\notin \mathcal{A}$.
\end{lemma}
\begin{proof}
    Let $p\in U$. If $p\in\R$, then $1/\hat{y}_G(p)\in\R$ and hence is by definition not contained in $\mathcal{A}.$
    Next suppose $p\notin\R$ and assume towards contradiction that for some two-terminal graphs $G,H$ we have $1/\hat{y}_G(p)\in \mathcal{A}_{H}$ and $p\notin \mathcal{E}(G)$.
    This means that $1<|\hat{y}_H(1/\hat{y}_G(p))|<\infty$ and moreover that $\hat{y}_H(1/\hat{y}_G(p))$ is not real.
    By Lemma~\ref{lem:substitution} we know that $\hat{y}_H(1/\hat{y}_G(p))=\hat{y}_{G(H)}(p)$ and therefore $p$ is contained in $\mathcal{A}_{G(H)}$, since $p\notin \mathcal{E}_{G(H)}$ because  ~\eqref{eq:edge replacement interaction} implies $R(G(H);p)+\SP(G(H);p)\neq 0$.
    This contradicts the fact that $U\cap \mathcal{A}=\emptyset$.
\end{proof}
The next result says something about the boundedness of the set $\mathcal{A}$.
\begin{proposition}\label{eq:density of A}
    The followings are equivalent
    \begin{enumerate}
        \item $\mathcal{A}$ is not dense in $\mathbb{C}$, \label{prop:eqivalence part1}
        \item $\mathcal{A}$ is bounded, i.e. there is an open $U\subseteq\widehat{\mathbb{C}}$ such that $\infty\in U$ and $U\cap \mathcal{A}=\emptyset$,\label{prop:eqivalence part2}
        \item there is an open set $O\subseteq \widehat{\mathbb{C}}$ such that $(1,\infty)\subseteq O$ and $O\cap \mathcal{A}=\emptyset$.
        \label{prop:eqivalence part3}
    \end{enumerate}
\end{proposition}

\begin{proof}
Observe that the implications (\ref{prop:eqivalence part3}) $\Rightarrow$ (\ref{prop:eqivalence part1}) and (\ref{prop:eqivalence part2}) $\Rightarrow$ (\ref{prop:eqivalence part1}) are true by definition. 

Let us prove (\ref{prop:eqivalence part1}) $\Rightarrow$ (\ref{prop:eqivalence part2}). Thus let us assume that $\mathcal{A}$ is not dense in $\mathbb{C}$, in particular there exists an open set $U\subseteq \mathbb{C}$, such that $U\cap \mathcal{A}=\emptyset$. 
As observed in the introduction we have that $\mathbb{D}\setminus \{0\}\subseteq \overline{\mathcal{A}}$ (because for the two terminal graph $K_2$ we have $\hat{y}_{K_2}(p)=1/p)$.
We may therefore assume that for any $p\in U$ we have $|p|>1$. 
Since $U$ is not empty and open, we can find an arc $\Gamma$ of a circle of radius $r>1$ with central angle $\theta>0$ in $U$ and such that there is $\varepsilon>0$ such that the $\varepsilon$-neighborhood of $\Gamma$ is still part of $U$.

Take $D_m=K_2^{\parallel m}$ for $m\ge 1$. Then $\frac{1}{\hat{y}_{D_m}(p)}=p^m$ and for any $m$ the set $U_m=\frac{1}{\hat{y}_{D_m}(U)}$ contains an open neighbourhood of an arc of central angle at least $m\theta$. More specificely, if $m$ is sufficiently large, the $m$th power of the $\varepsilon$-neighborhood of $\Gamma$ become an annulus, thus showing that
\[
 U'_m:=\{p\in\mathbb{C}~|~(r-\varepsilon)^m<|p|<(r+\varepsilon)^m\}\subseteq U_m.
\]
By the previous lemma each $U_m$ is disjoint from $\mathcal{A}$, since the exceptional sets of the parallel compositions $K_2^{\parallel m}$ are empty.
We claim that $U'=\bigcup_{m\in\mathbb{N}}U'_m \cup\{\infty\}$ will witness the boundedness of $\mathcal{A}$. First observe that if $m$ is sufficiently large, then $(r+\varepsilon)^m>(r-\varepsilon)^{m+1}$ and thus the union of the annuli $U'_{m}, U'_{m+1}$ form a connected annulus. As $r+\varepsilon>1$, we obtain that $\cup_{m\in\mathbb{N}}U'_m\supseteq \mathbb{C}\setminus B(0,M)$ for some $M>0$ and is disjoint from $\mathcal{A}$. This shows that $\mathcal{A}$ is a bounded set, proving (\ref{prop:eqivalence part1}) $\Rightarrow$ (\ref{prop:eqivalence part2}).

Now let us prove (\ref{prop:eqivalence part2}) $\Rightarrow$ (\ref{prop:eqivalence part3}), i.e. assume that $\mathcal{A}$ is bounded. Let $M>1$ such that for all $p\in \mathcal{A}$, $|p|<M$ and let $U=\{p\mid |p|>M\}$.
    Observe that for $G_1=K_2$ we have $1/\hat y_{K_2}(p)=p$ so $1/\hat y_{K_2}(U)=U$. 
    Consider for $n\in \mathbb{N}$, $G_n=K_2^{\series_n}$ and the map $f_n$ defined by $p\mapsto 1/\hat{y}_{G_n}(p)$.
  By Lemma~\ref{lem:zero-free} it follows that, 
  \begin{align}\label{eq:statement f_n(U)}
\text{ for all $n$}\quad    f_n(U\setminus \mathcal{E}(K_2^{\bowtie_n}))\cap \mathcal{A}=\emptyset.
  \end{align}
 The exceptional set of $K_2^{\bowtie_n}$ consists of those $p$ for which $R(K_2^{\bowtie_n};p)+\SP(K_2^{\bowtie_n};p)=0$, or equivalently those $p$ for which $\SP(K_2^{\bowtie_n};p)\neq 0$ and $\hat{y}_{K_2^{\bowtie_n}}(p)=0$ or $R(K_2^{\bowtie_n};p)=0=\SP(K_2^{\bowtie_n};P)=0$.
    Now since $R(K_2^{\bowtie_n};p)=R(K_2;p)^n=(1-p)^n$ only has $p=1$ as zero and $\hat{y}_{K_2^{\bowtie_n}}(p)=\frac{1+(n-1)p}{np}$, by Lemma~\ref{lem:multi-series}, it follows that 
    \begin{align}\label{eq:statement exceptional set}
 \text{ for all $n$}\quad    \mathcal{E}(K_2^{\bowtie_n})\subseteq[-1,1].     
    \end{align}
This implies that the exceptional sets $\mathcal{E}(K_2^{\bowtie_n})$ are disjoint from $U$ and hence by~\eqref{eq:statement f_n(U)} we have
\begin{align}\label{eq:statement f_n(U) improved}
\text{ for all $n$}\quad   f_n(U)\cap \mathcal{A}=\emptyset.
  \end{align}
By Lemma~\ref{lem:multi-series} we have $f_n(p)=\frac{n}{1/p+n-1}=\frac{pn}{1+p(n-1)}$ and thus $f_n$ is a M\"obius transformation with real coefficients and hence preserves the real line.  
  The set $U_1:=U$ contains the real interval $(M,\infty]$, and since $f_n$ preserves orientation (as its derivative is positive) it follows that $U_n:=f_n(U)$ contains the real interval 
    $$\left(f_{n}(M),{f_{n}(\infty)}\right]=\left(\frac{nM}{1+nM-M},\frac{n}{n-1}\right].$$
    For $n>M+2$ we have $\frac{nM}{1+nM-M}<\frac{n+1}{n}$ and therefore 
    $U_n\cap \mathbb{R}$ and $U_{n+1}\cap \mathbb{R}$ have a nonempty intersection.
    Since $\frac{nM}{1+nM-M}\to 1$ as $n\to \infty$, it follows that 
    $O_1:=\bigcup_{n\in\mathbb{N}}U_n$ contains an interval of the form $(1,w)$ for some $w>1$. 
  By~\eqref{eq:statement f_n(U) improved} $O_1$ is disjoint from $\mathcal{A}$. 
    
  Next take $D_m=K_2^{\parallel m}$. Then $\frac{1}{\hat{y}_{D_m}(p)}=p^m$ and $O_m:=\frac{1}{\hat{y}_{D_m}(O_1)}$ contains an interval of the form $(1,w^m)$ and is disjoint from $\mathcal{A}$ by the previous lemma, since the exceptional sets of the parallel compositions $K_2^{\parallel m}$ are empty.
    Thus $O=\bigcup_{m\in\mathbb{N}}O_m$ is an open set containing $(1,\infty)$ such that $O\cap \mathcal{A}=\emptyset$, as desired.
\end{proof}

We can now provide a proof of Proposition~\ref{prop:unbounded zeros?}

\begin{proof}[Proof of Proposition~\ref{prop:unbounded zeros?}]
By the previous proposition, the closure of $\mathcal{A}$ is equal to $\mathbb{C}$ if and only if  $\mathcal{A}$ is unbounded, if and only if there exists $p>1$ such that $p\in \overline{\mathcal{A}}$.
The proposition now follows by Lemma~\ref{lem:closure of all activity and zeros}.
\end{proof}

\subsection{Activity locus and the unit circle}
In this subsection we aim to show that all but finitely many points of the closed unit circle are part of the activity locus. 
To do so, let us recall the multivariate version of the edge replacement construction see Lemma~\ref{lemma:multivariate edge replacement}. Define a template graph $(H,c)$ to be a two-terminal graph $H$ with edge label  $c:E(H)\to\{a_1,a_2\}$. For any $G_1,G_2$ two-terminal graphs let $H(G_1,G_2)$ be the graph obtained by replacing each edge labeled by $a_i$ with $G_i$. This is a special case of the construction described in Lemma~\ref{lemma:multivariate edge replacement}.

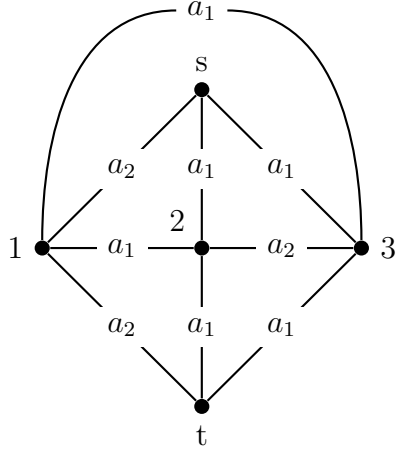
\begin{figure}[ht]
    \centering
    \begin{tikzpicture}[scale=0.7]
    \node[circle,fill,inner sep=2pt,label=above:s] (s) at (0,3) {};
    \node[circle,fill,inner sep=2pt,label=left:1] (1) at (-3,0) {};
    \node[circle,fill,inner sep=2pt,label=135:2] (2) at (0,0) {};
    \node[circle,fill,inner sep=2pt,label=right:3] (3) at (3,0) {};
    \node[circle,fill,inner sep=2pt,label=below:t] (t) at (0,-3) {};
    \node[fill=white] (a) at (0,4.5) {$a_1$};
    \draw[thick] (s) -- (1) node [midway,fill=white] {$a_2$};
    \draw[thick] (s) -- (2) node [midway,fill=white] {$a_1$};
    \draw[thick] (s) -- (3) node [midway,fill=white] {$a_1$};
    \draw[thick] (1) edge[out=90,in=180,-] (a);
    \draw[thick] (a) edge[out=0,in=90] (3);
    \draw[thick] (1) -- (2) node [midway,fill=white] {$a_1$};
    \draw[thick] (1) -- (t) node [midway,fill=white] {$a_2$};
    \draw[thick] (2) -- (3) node [midway,fill=white] {$a_2$};
    \draw[thick] (2) -- (t) node [midway,fill=white] {$a_1$};
    \draw[thick] (3) -- (t) node [midway,fill=white] {$a_1$};
\end{tikzpicture}
    \caption{This is the template graph used in Lemma~\ref{lemma:pentagonal}. If $G_1,G_2$ are two-terminal graphs, then by replacing each $a_i$ edge with $G_i$, the obtained graph is denoted by $H(G_1,G_2)$.}
    \label{fig:pentagonal}
\end{figure}

Let us define the pentagon template $H$ as in Figure~\ref{fig:pentagonal}. 
\begin{lemma}\label{lemma:pentagonal}
    Let $G_1,G_2$ be two two-terminal graphs and let $(H,c)$ be the pentagon template from Figure~\ref{fig:pentagonal}. Then
    \[
    \hat{y}_{H(G_1,G_2)}(p)=F(\hat y_{G_1}(p),\hat y_{G_2}(p)),
    \]
    where
    \begin{equation*}
    F(y_1,y_2)\!=\!\frac{{\left(y_{1}^{5}\!+\!y_{1}^{4}\!+\!y_{1}^{3}\!+\!y_{1}^{2}\!+\!y_{1}\!+\!1\right)  } y_{2}^{3} \!-\! 2  {\left(y_{1}^{2}\!+\!y_{1}\!+\!1\right)} y_{2}^{2} \!-\! 2  y_{1}^{2} \!-\! {\left(y_{1}^{3}\!+\!y_{1}^{2}\!+\!2  y_{1}\!+\!2\right)} y_{2}\!+\!2  y_{1}\!+\!6}{2  {\left({\left(y_{1}^{3}\!+\!y_{1}^{2}\!+\!2  y_{1}\!+\!2\right)} y_{2}^{2}\!+\!{\left(2  y_{1}^{2} \!-\! 5  y_{1} \!-\! 9\right)} y_{2} \!-\! 6  y_{1}\!+\!12\right)}}
    \end{equation*}
    In particular,
    \[
    |F(e^{it},e^{-it})|^2=\frac{8 \, \sin\left(t\right)^{4} + \sin\left(t\right)^{2}}{2 \, {\left(12 \, \cos\left(t\right)^{2} - 25 \, \cos\left(t\right) + 13\right)}}.
    \]
    
\end{lemma}
\begin{proof}
    The proof is a direct application of Lemma~\ref{lemma:multivariate edge replacement} and the computation is straightforward; we include a Sage code~\cite{sagemath} for the computation in the \hyperlink{code:pentagon}{Appendix}.
\end{proof}








To prove Proposition~\ref{prop:zeros vs active} we need the following technical lemma about $F(y_1,y_2)$ defined in the previous lemma.
\begin{lemma}
For the function $F$ defined in the previous lemma we have
\[
|F(e^{it},e^{-it})|^2\ge 1
\]
if and only if $|cos(t)|\ge \tfrac{1}{4} (5 \sqrt{2} - 4)$.
\end{lemma}
\begin{proof}
By the previous lemma we know that
\[
|F(e^{it},e^{-it})|^2=\frac{8 \, \sin\left(t\right)^{4} + \sin\left(t\right)^{2}}{2 \, {\left(12 \, \cos\left(t\right)^{2} - 25 \, \cos\left(t\right) + 13\right)}}=\frac{(1-\cos(t) ) (\cos(t) + 1) (9-8 \cos^2(t))}{2(1-\cos(t) ) (13-12 \cos(t))}.
\]
Since $13-12\cos(t)>0$, $|F(e^{it},e^{-it})|^2\ge 1$ if and only if
\begin{eqnarray*}
(\cos(t) + 1) (9-8 \cos^2(t)) -2(13 -12 \cos(t) )&\ge 0,\\
\Leftrightarrow -8 \cos^3(t) - 8 \cos^2(t) + 33 \cos(t) - 17 &\ge 0,\\
\Leftrightarrow -(\cos(t) - 1) (8 \cos^2(t) +16\cos(t) - 17) &\le 0.
\end{eqnarray*}
Note that the polynomial $(x-1)(8x^2+16x-17)\le 0$ if and only if $\tfrac{1}{4} (5 \sqrt{2} - 4)\le x\le 1$ or $x\le -\tfrac{1}{4} (4 + 5 \sqrt{2})<-1$. Since  $\cos(t)$ takes values from $[-1,1]$ we obtain the desired statement.
\end{proof}

\begin{proposition}
    Let $p\in\mathbb{C}$ such that $|p|=1$ and $p^k\neq 1$ for $k=1,\dots,9$. Then
    \[
    p\in\mathcal{A}.
    \]
\end{proposition}
\begin{proof}
    Let $\arg(p)=2\pi\alpha$ for some $\alpha\in(0,1]$ and denote $I=\{t\in(-\pi,\pi]~|~\cos(t)> \tfrac{1}{4} (5 \sqrt{2} - 4) \}$. 

    First, let us assume that $\alpha=\frac{n}{m}$ is rational, where $\gcd(n,m)=1$. Since $\cos(2\pi/10)>\frac{1}{4}(5\sqrt{2}-4)$, we know that $m\ge 10$. This means that there exists a $k,\ell\in\mathbb{N}$ such that $1\neq \arg(p^{-k})\in I$ and $p^{-k\ell}=p^{k}$. If we let $G_1=K_2^{\parallel k}$ and $G_2=K_2^{\parallel k\ell}$, then
    \[
    |\hat y_{H(G_1,G_2)}(p)|=|F(\hat y_{G_1}(p),\hat y_{G_2}(p))|=|F(p^{-k},p^{-k\ell})|=|F(p^{-k},p^{k})|>1,
    \]
    by the previous lemma.
     Since the exceptional set of $K_2^{m\parallel}$ is empty for any $m\geq 1$, this shows that $p\in\mathcal{A}$.

    Now let us assume that $\alpha$ is irrational. 
    In what follows we use the fact that for any irrational angle $\beta$ the orbit $\{e^{k(2\pi i \beta)}\mid k\in \mathbb{N}\}$ is dense in the unit circle.
    Therefore there exists a $k\in\mathbb{N}$ such that $0\neq \arg(p^{-k}) \in  I$.
    Thus by the previous lemma $|F(p^{-k},p^k)|>1$.
    Since $F(y_1,y_2)$ is a continuous function, there exists a $\delta>0$ such that $|F(p^{-k},y_2)|>1$ for all $y_2\in B_\delta(p^{k})$. 
    Since $-k\alpha$ is irrational there exists $\ell\in\mathbb{N}$ such that $(p^{-k})^\ell \in B_\delta(p^{k})$. This means that $|f(p^{-k},p^{-k\ell})|>1$. Now let $G_1=K_2^{\parallel k}$ and $G_2=K_2^{\parallel k\ell}$, then
    \[
    |\hat y_{H(G_1,G_2)}(p)|=|F(\hat y_{G_1}(p),\hat y_{G_2}(p))|=|F(p^{-k},p^{-k\ell})|>1,
    \]
    showing that $p\in \mathcal{A}$.
\end{proof}

In the next proposition we report, further points from the unit circle that are contained in the activity locus. 
\begin{proposition}\label{prop:5to9}
    For $k\in\{5,6,7,8,9\}$ there exists a two-terminal graph $G_k$, such that 
    \[
        |\hat y_{G_k}(e^{2\pi i/k})|>1 \textrm{ and } \gcd(R(G_k;p),p^k-1)=p-1.
    \]
    In particular, $e^{2\pi i/k}\in \mathcal{A}$.
\end{proposition}
\begin{proof}
    Let $G_k$ be the graph given by the adjacency matrix $A_k$, where
\[
A_9=\left(\begin{array}{ccccc}
     0& 0& 1& 1 & 8 \\
     0& 0& 1& 1 & 8 \\
     1& 1& 0& 8 & 2 \\
     1& 1& 8& 0 & 2 \\
     8& 8& 2& 2 & 0 \\
\end{array}\right)\qquad A_8=\left(\begin{array}{ccccc}
     0& 0& 1& 1 & 7 \\
     0& 0& 1& 1 & 7 \\
     1& 1& 0& 7 & 2 \\
     1& 1& 7& 0 & 2 \\
     7& 7& 2& 2 & 0 \\
\end{array}\right) \qquad A_7\left(\begin{array}{ccccc}
     0& 0& 1& 1 & 6 \\
     0& 0& 1& 1 & 6 \\
     1& 1& 0& 6 & 2 \\
     1& 1& 6& 0 & 2 \\
     6& 6& 2& 2 & 0 \\
\end{array}\right)
\]
\[
A_6=\left(\begin{array}{cccccc}
     0& 0& 1& 1 & 1 & 5 \\
     0& 0& 1& 1 & 1 & 5 \\
     1& 1& 0& 5 & 5 & 5 \\
     1& 1& 5& 0 & 5 & 2 \\
     1& 1& 5& 5 & 0 & 2 \\
     5& 5& 5& 2 & 2 & 0 \\
\end{array}\right)\qquad A_5=\left(\begin{array}{cccccc}
     0& 0& 1& 1 & 1 & 4 \\
     0& 0& 1& 1 & 1 & 4 \\
     1& 1& 0& 4 & 4 & 2 \\
     1& 1& 4& 0 & 4 & 2 \\
     1& 1& 4& 4 & 0 & 3 \\
     4& 4& 2& 2 & 3 & 0 \\
\end{array}\right)
\]
By choosing the two non-adjacent vertices of each $G_k$ to be its terminals, one can verify the claim about the $G_k$. We have included a Sage code~\cite{sagemath} which verifies this in the \hyperlink{code:5to9}{Appendix}. 

To show the second part, we have to show that for $z_k=e^{\frac{2\pi i}{k}}\notin \mathcal{E}(G_k)$, i.e. $R(G_k,z_k)+S(G_k,z_k)\neq 0$. We know that $z_k^k-1=0$ and $z_k\neq 1$, therefore $R(G_k;z_k)\neq 0$. On the other hand,
\[
    |R(G_k,z_k)+S(G_k,z_k)|=|R(G_k,z_k)|\cdot |1+\hat y_{G_k}(z_k)| \ge |R(G_k,z_k)|\cdot (|\hat y_{G_k}(z_k)|-1) > 0,
\]
which proves that $z_k\notin\mathcal{E}(G_k)$ as we desired.











\end{proof}

We now prove Proposition~\ref{prop:zeros outside disk} from the introduction.
\begin{proof}[Proof of Proposition~\ref{prop:zeros outside disk}]
Let $p\in \mathbb{C}$ of norm $1$ such that $p^k\neq 1$ for $k\in \{1,\ldots,4\}$.
By the previous propositions we know $p\in \mathcal{A}$ and since this is a open set, there exists $\varepsilon>0$ such that $B(p,\varepsilon)\subset \mathcal{A}$.
The result now follows since $\overline{\mathcal{A}}=\overline{\mathcal{Z}}$, by Lemma~\ref{lem:closure of all activity and zeros}.
\end{proof}


\section{Density implies hardness}\label{sec:density implies hardness}
In this section we will prove Theorem~\ref{thm:main hard} following the proof of~\cite{main-approx}*{Theorem 3.12}.
We will start with a proof outline after which we will gather the ingredients as discussed in this outline.

To prove hardness of approximately computing $R(F;p)$ for (planar) graphs $F$ when $y\in \mathcal{D}_G$, we show a polynomial-time reduction from approximation to exact computation of $R(F;p)$, which is known to be \textsc{\#P}-hard by a result of Vertigan~\cite{vertigan-2005}.
An outline of the reduction algorithm for the case planar graphs is roughly as follows.

We assume that we have access to an oracle for \textsc{Approx-Abs-Planar-Rel$(p)$} or \textsc{Approx-Arg-Planar-Rel$(p)$}.
Rather than trying to compute $R(F;p)$ directly under this assumption, we first try to compute ratios of the form 
\begin{equation*}
\frac{R(F;p)}{R(F\setminus e;p)}
\end{equation*}
for an edge $e$ of $F$ and combine these in a telescoping fashion to compute $R(F;p)$ exactly.
We do this by viewing the ratio as the solution $x^*$ to a linear equation of the form $Ax-B=0$ with $B=R(F;p)$ and $A=R(F\setminus e;p)$ and $x=y_{G'}-(p+1)$ for some $G'\in \Hs_{G}$ such that $R(G';p)\neq 0$. 
Making use of Lemma~\ref{lem:con-del}, we can view $Ax+B$ as the reliability polynomial of $F(G')_e$ (the graph obtained by implementing $G'$ on the edge $e$ of $F$) and use the oracle to approximately determine the value of $Ax+B$.
By using different values of $x$ that we can achieve using the fact that $p$ is contained in the density locus of $\mathcal{H}_{G}$, in combination with a form of binary search we can then determine $x^*$ with very high precision.
This requires us to actually generate any given value $x$ with very high precision as $y_{G'}-(p+1)$ for some $G'\in \Hs_{G}$ such that $R(G';p)\neq 0$ in polynomial time.
Using the fact that algebraic numbers of bounded complexity form a discrete set (much like the rational numbers of bounded denominator) we can then determines the value $x^*$ exactly using an algorithm due to Kannan, Lenstra and Lov\'asz~\cite{kannan-1988}.

There is a mild caveat to the above approach. 
Namely, if both $R(F;p)=0$ and $R(F\setminus e;p)=0$, then $R(F(G')_e;p)=0$ and we cannot `trust' the oracle. 
However by doing the telescoping procedure with a bit more care, we can sidestep this issue and finally combine everything to compute $R(F;p)$ exactly.

In Section~\ref{sec:exponential density} we devise an algorithm to get arbitrarily close to any given point with an effective edge interaction.
In Section~\ref{sec:ratio computing} we show how to compute the ratios exactly when given an oracle for \textsc{Approx-Abs-Planar-Rel$(p)$} or \textsc{Approx-Arg-Planar-Rel$(p)$}. Finally, in Section~\ref{sec:telescoping} we complete the telescoping argument.

Before we get started we first recall some basic facts about algebraic numbers and recall the result of Vertigan~\cite{vertigan-2005} about the complexity of exactly computing the reliability polynomial.

\subsection{Representing algebraic numbers} 
We collect here some basic properties of algebraic numbers and how to represent them following~\cite{main-approx}.

By definition an \emph{algebraic number} is a complex number $\alpha$ that is a root of a polynomial with integer coefficients.
The minimal polynomial of an algebraic number $\alpha$ is the unique polynomial $q(x)\in\mathbb{Z}[x]$ of smallest degree such that $q(\alpha)=0$, whose coefficients have no common prime factors, and whose leading coefficient is positive.

In this paper we will represent an algebraic $\alpha$ number as a pair $(q,R)$ where $q\in\mathbb{Z}[x]$ is the minimal polynomial of $\alpha$ and $R$ is an open rectangle in the complex plane such that $\alpha$ is the only zero of $p$ in that rectangle. 
A typical implementation is a list of numbers representing the coefficients of the polynomial and a 4-tuple of rational numbers $(a,b,c,d)$ representing the rectangle $(a,b)\times(c,d)\subseteq\R^2\cong\C$. 
This representation is of course not unique, but one can decide in polynomial time whether two representations represent the same number by an algorithm due to Wilf~\cite{wilf-1978}.

We define the \emph{size} of a representation of an algebraic number $\alpha$ given as a pair $(q,R)$ as the number of bits required to represent the polynomial $q$ and the rectangle $R$. 
We can perform basic operations (addition, subtraction, multiplication, division and integer root) on the representations in time polynomial in the size of the representation; see~\cite{strzebonski-1997} and~\cite{main-approx}*{Section 2.3}.

We need a few more definitions. Let $q=\sum_{i=0}^d a_i x^i$ be a polynomial of degree $d$, with leading coefficient $a_d$ and roots $\alpha_1,\ldots,\alpha_d$.
\begin{itemize}
    \item We define the \emph{usual height} of $q$ as its maximum coefficient by absolute value and denote it by $H(q)$.
    \item We define the \emph{length} of $q$ as the sum of the absolute value of the coefficients and denote it by $L(q)$.
    \item  We define the \emph{absolute logarithmic height} of $q$ as
    $$h(q):=\frac{1}{d}\log\left(|a_d|\prod_{i=1}^d\max(1,|\alpha_i|)\right).$$
\end{itemize}
Similarly, we define the length, usual height and absolute logarithmic height of an algebraic number as the length, usual length and absolute logarithmic height of its minimal polynomial.

\subsection{Exact Computation}
We define the problem of exactly evaluating the reliability polynomial of a graph $G$ for a fixed algebraic number $p$:
\\ \quad \\
\begin{tabular}{rl}
    Name: & \textsc{Planar-Rel$(p)$} \\
    Input: & A planar graph $H$.\\
    Output: & A representation of the algebraic number $R(H;p)$.
\end{tabular}
\\ \quad \\
The problem \textsc{Rel$(p)$} is defined in the same way, except that the input can now be any graph $H$.

Let us recall the definition of the Tutte polynomial of a graph $G=(V,E)$, 
\[
T(G;x,y):=\sum_{A\subseteq E}(x-1)^{k(A)-k(E)}(y-1)^{|A|-|V|+k(A)}.
\] 
We note that $R(G;p)$ is equal to $T(G;1,1/p)$ up to a simple transformation if $G$ is connected, and is identically zero otherwise. 
Indeed, if $G$ is connected, we have
\[
\text{Rel}(p)=\left(\frac{1-p}{p}\right)^{|V|-1}p^{|E|}T(G;1,1/p).
\]
Since connectedness can be verified in polynomial time by Breadth First Search for example, \textsc{Planar-Rel$(p)$} reduces trivially to \textsc{Planar-Tutte$(1,1/p)$}, where for algebraic numbers $x,y$, \textsc{Planar-Tutte$(x,y)$} is defined as follows.
\\ \quad \\
\begin{tabular}{rl}
    Name: & \textsc{Planar-Tutte$(x,y)$} \\
    Input: & A planar graph $G$.\\
    Output: & A representation of the algebraic number $T(G;x,y)$.
\end{tabular}
\\ \quad \\
A result of Vertigan~\cite{vertigan-2005}*{Proposition 4.4 (iii)} saying that \textsc{Planar-Tutte$(x,y)$} is \textsc{\#P}-hard for most algebraic numbers $x,y$ directly implies the following:
\begin{theorem}[Vertigan, 2005]\label{thm:vertigan}
The problem \textsc{Planar-Rel$(p)$} is \textsc{\#P}-hard for any algebraic number $p\notin \{0,1\}$. 
\end{theorem}
This theorem directly implies that \textsc{Rel$(p)$} is \textsc{\#P}-hard for any algebraic number $p\notin \{0,1\}$, a result obtained earlier in~\cite{JVWhardJonesandTutte}.
\subsection{Exponential density}\label{sec:exponential density}
In this section we prove the following result, analogous to \cite{main-approx}*{Theorem 3.6}, which allows us to efficiently approximate arbitrary points of $\Q[i]$ with effective edge interactions.

For a graph $G=(V,E)$ we define its \emph{size} as $|G|:=|V|+|E|$.

\begin{theorem}\label{thm:constructing}
Let $G_0$ be a two-terminal graph and let $p\in \mathcal{D}_{G_0}$ (resp. $p\in \mathcal{D}^{\mathbb{R}}_{G_0}$) such that $R(G_0;p)\neq 0$ and $S(G_0;p)\neq 0$. 
    Then there exists an algorithm that on input of $y_0\in\Q[i]$ (resp. $y_0\in\Q$) and rational $\epsilon>0$ outputs a two-terminal graph $G\in\Hs_{G_0}$ and the value $R(G;p)$ satisfying $|y_G(p)-(p+1)-y_0|<\varepsilon$ and $R(G;p)\neq 0$. 
    Both the running time of the algorithm and the size of $G$ are $\poly(\size(y_0,\epsilon))$.
\end{theorem}
To prove the theorem we require a few preliminary results.

We call a M\"obius transformation $\Phi$ \emph{contracting} on a set $U\subseteq\C$ if there exists $\alpha\in (0,1)$ such that $|\Phi'(z)|<\alpha$ for all $z\in U$. 
We restate a slightly modified version of Lemma 2.8 from~\cite{Bezakovahardcore}. 
\begin{lemma} \label{lem:4.5}
    Let $m\in\Q[i]$, $r>0$ rational and $U=B(m,r)$. 
    Suppose that we have M\"obius transformations with algebraic coefficients $\Phi_i:\Cext\to\Cext$ for $i\in[\ell]$ satisfying the following:
    \begin{itemize}
        \item[(a)] for each $i\in[\ell]$, $\Phi_i$ is contracting on $B(m,3r)$,
        \item[(b)] $U\subseteq \bigcup_{i\in [\ell]}\Phi_i(U)$.
    \end{itemize}
    Then there is an algorithm which on input of algebraic numbers $s,t\in U$ (respectively the starting point and target) and rational $\epsilon>0$ outputs in $\poly(\size(s,t,\epsilon))$-time an algebraic number $x\in B(t,\epsilon)$ and a sequence $i_1,\ldots,i_k\in[\ell]$ such that
    \begin{itemize}
        \item[(i)] $k\in O(\log(\epsilon^{-1}))$
        \item[(ii)] $x=\Phi_{i_1}\circ\cdots\circ\Phi_{i_k}(s)$ and
        \item[(iii)] $\Phi_{i_j}\circ\cdots\circ\Phi_{i_k}(s)\in B(m,3r)$ for all $j\leq k$. 
    \end{itemize}
Moreover the same is true when $m\in \Q$, $U=(m-r,m+r)$ and the coefficients of the $\Phi_i$ are real (and algebraic).
\end{lemma}
\begin{remark}\label{rem:proof of lemma is wrong}
We note that item (iii) in the lemma above is not part of the statement of~\cite{Bezakovahardcore}*{Lemma 2.8}. Additionally, we require the $\Phi_i$ to be contracting on $B(m,3r)$ rather than on $U=B(m,r)$.
The reason for this restatement is that we believe that the proof given in~\cite{Bezakovahardcore} is not correct, since it implicitly assumes that all values $\Phi_{i_j}\circ\cdots\circ\Phi_{i_k}(s)$ are contained in $U$ to be able to apply the bound on the derivative. This however does not have to be the case.
We note that our restatement can be applied in essentially the same way as it was originally applied in~\cite{Bezakovahardcore} and later in~\cite{GalanisetalboundeddegreeIsing} and~\cite{main-approx}. Therefore the error in the proof does not have any implications for the contents of these papers.
\end{remark}
\begin{proof}
We follow the ideas of the proof of~\cite{Bezakovahardcore}*{Lemma 2.8} and modify some of the steps appropriately.
Let $M\in (0,1)$ be an upper bound on $|\Phi'(z)|$ for all $i=1,\ldots, \ell$ and $z\in B(m,3r)$.
Set $t_0=t$, then recursively define $t_j$ for $j\geq 1$ by choosing an index $i_{j}$ such that $\Phi_{i_j}^{-1}(t_{j-1})\in U$ and letting $t_{j}=\Phi_{i_j}^{-1}(t_{j-1})$.
At the moment when $s\in B(t_k,\varepsilon/M^{-k})$ we stop.

We note that once $k$ is such that $\varepsilon M^{-k}\geq 2r$ we have $s\in B(t_k,\varepsilon/M^{-k})$, since both $t_k,s\in U=B(m,r)$. This implies part (i).

To see parts (ii) and (iii).
Let us denote $s_k=s$ and for $j=k-1, \ldots, 0$ let $s_{j}=\Phi_{i_{j+1}}(s_{j+1}).$
We claim that for all $j=0,\ldots,k$ we have $|s_j-t_j|\leq 2r$ implying $s_j\in B(m,3r)$ as well as $|s_j-t_j|\leq \varepsilon/M^{-j}$. Together this implies parts (ii) and (iii).
We prove the claim by descending induction. For $j=k$ the two claims are true by construction. Now assume $j<k$.
Since $s_{j+1},t_{j+1}$ are both contained in $B(m,3r)$, we have by the fundamental theorem of calculus
\[
|s_j-t_j|=|\Phi_{i_{j+1}}(s_{j+1})-\Phi_{i_{j+1}}(t_{j+1})|\leq \sup_{z\in B(m,3r)}|\Phi_{i_{j+1}}'(z)|\cdot |s_{j+1}-t_{j+1}|\leq M|s_{j+1}-t_{j+1}|.
\]
So by induction we have $|s_j-t_j|\leq \varepsilon M^{-j}$ and $|s_j-t_j|\leq M2r<2r$.
This finishes the proof.

A similar proof also works in the special case when all parameters are real.  
\end{proof}

The next lemma tell us how to construct the M\"obius transformations using effective edge interactions of two-terminal graphs that satisfy the hypothesis of Lemma~\ref{lem:4.5}.

Let us define for $p\in \mathbb{C}$ and a two-terminal graph $G_0$ the set
\begin{equation}
    \Hs_{G_0}^*(p):=\{G\in \Hs_{G_0}\mid R(G;p)\neq 0\}.
\end{equation}

\begin{lemma}\label{lem:precomputation}
Let $G_0$ be a two-terminal graph and let $p\in \mathcal{D}_{G_0}$ (resp. $p\in \mathcal{D}^{\mathbb{R}}_{G_0}$) such that $R(G_0;p)\neq 0$.
Then there exists $r>0$ and  graphs $G_1,\dots,G_\ell\in \Hs^*_{G_0}(p)$, such that the M\"obius transformations $\Phi_i$ for $i=1,\ldots \ell$ defined as $z\mapsto \frac{z\YGz-p}{\YGz+z-1-p}+y_{G_i}(p)-1$  satisfy the  hypothesis of Lemma~\ref{lem:4.5} for $U=B(1,r)$ (resp. $U=(m-r,m+r)$).
\end{lemma}
\begin{proof}
Define the M\"obius transformation $g$ as in~\eqref{eq:define mobius}, that is, 
\begin{align*}
    g(z)&=f_p(f_p(z)f_p(\YGz))=\frac{z\YGz-p}{\YGz+z-1-p}.\\
\end{align*}
By Theorem~\ref{prop:active vs density} we have $p\in \mathcal{A}_{G_0}$ (resp. $p\in \mathcal{A}^{\mathbb{R}}_{G_0}$)  and thus by Lemma~\ref{lem:mobius} there exists $\alpha\in (0,1)$ such that $|g'(1)|<1-\alpha$. 

We now given the proof for the case that $p\notin\mathbb{R}$, remarking that it also applies to the real setting.

Consider the open set $U_1=\{z\in\C\mid |g'(z)|<1-\alpha/2\}$. 
Clearly, $1\in U_1$. 
Since $U_1$ is open there exists $r>0$ such that $B(1,3r)\subseteq U_1$. (We can easily determine such $r$ explicitly from the formula for $g'(z)$.)
Then for every $u\in B(1,3r)$ we have that $z\mapsto g(z)+u-1$ has the absolute value of its derivative bounded by $1-\alpha/2$ on $B(1,3r)$. 
Let $U=B(1,r)$.
Since $\{y_G(p)\mid G\in \Hs_{G_0}, R(G;p)\neq 0\}$ is dense in $\mathbb{C}$ the collection of open sets of the form $g(U)+y_{G}(p)-1$ with $G\in \Hs^*_{G_0}(p)$ covers the compact set $\overline{U}$.
Therefore there exists a finite set of graphs $G_i\in \Hs^*_{G_0}(p)$, $i=1,\ldots,t$ such the associated sets cover $\overline{U}$. 

Therefore, the corresponding set of contracting M\"obius transformations $\{\Phi_i(z)=g(z)+y_{G_i}(p)-1\}_{i\in [t]}$  satisfies the hypothesis of Lemma~\ref{lem:4.5}.
\end{proof}

We are now ready to prove Theorem~\ref{thm:constructing}
\begin{proof}[Proof of Theorem~\ref{thm:constructing}]\label{proof:constructing}
Let $p,G_0$ as in the theorem statement. 
We give a proof for the case $p\notin \mathbb{R}$, along the way remarking why it also applies to the case when $p\in \mathbb{R}$.

First let us collect quantities that are used in our algorithm, but don't depend on the input $\varepsilon>0$ and $y_0$.

By the previous lemma there exists $r>0$ and M\"obius transformations $\Phi_i$, $i=1,\ldots ,\ell$ of the form 
\begin{align*}
z\mapsto \frac{z\YGz-p}{\YGz+z-1-p}+y_{G_i}(p)-1=g(z)+y_{G_i}(p)-1
\end{align*}
with $G_i\in \Hs^*_{G_0}(p)$ that satisfy the hypothesis of Lemma~\ref{lem:4.5} with $U=B(1,r)$. 
Let us recall that $1$ is an attracting fixed point of the M\"obius transformation $g$ as defined in~\eqref{eq:define mobius}, while $p$ is the other fixed point, which is repelling, i.e. $|g'(p)|>1$.
Also, recall that if $z=y_H(p)$ for some graph $H$ then $g(z)=y_{H\parallel G_0}(p)$. 
Therefore 
\begin{equation}\label{eq:composition Phi}
\Phi_i(y_H(p))=y_{H\parallel G_0}(p)+y_{G_i}(p)-1=y_{(H\parallel G_0)\series G_i}(p).
\end{equation}


Let us fix $N\in \mathbb{N}$, such that $\bigcup_{i=1}^N g^{\circ i}(B(p,2))$ contains $\widehat{\mathbb{C}}\setminus U$.
The existence of $N$ is guaranteed by the fact that $p$ is a repelling fixed point of $g$, which implies that for any open set $O$ containing $p$ we have $=\widehat{\mathbb{C}}\setminus \{1\}\subseteq \bigcup_{i=1}^\infty g^{\circ i}(O)$.
This follows for example from the fact that if we conjugate $g$ with $h:z\mapsto \frac{z-1}{z-p}$ we have that $h(p)=\infty$ and $h\circ g\circ h^{-1}(z)=\alpha z$, where $\alpha=g'(1)$, which has magnitude less than 1 (see~\cite{beardon}, and also later in this proof for more details).
(In case $p\in \mathbb{R}$ we have that $g$ has real coefficients, and hence $\bigcup_{i=1}^N g^{\circ i}(B_\mathbb{R}(p,2))$ contains $\widehat{\mathbb{R}}\setminus U$, where $U=(1-r,1+r)$.) 


Next, we set $s\in U$ such that there is an $n_s\in\mathbb{N}$ such that
\[
s:=y_{G_0^{\parallel n_s}}(p)\in U.
\]
The existence of $n_s$ is guaranteed by the proof of Proposition~\ref{prop:active vs density}, which showed $y_{G_0^{\parallel n}}(p)\to 1$ as $n\to\infty$. Note that the proof of Proposition~\ref{prop:active vs density} also implies that we may assume $R(G_0^{\parallel n_s};p)\neq 0$.

Now we are ready to describe the desired algorithm. Let $\omega_0=y_0+(1+p)$. Our goal is to find an effective edge interaction $\omega_1=y_G(p)$ such that $|\omega_1-\omega_0|\leq \varepsilon$ in polynomial time of the input.
For the algorithm we consider three different cases: (1) $ \omega_0\in U$, (2) $\omega_0\in B(p,2)\setminus U$ and (3) $\omega_0\in \mathbb{C}\setminus (U\cup B(p,2))$. 

\begin{enumerate}[label=Case (\arabic*)]
\item In this case we can run the algorithm from Lemma~\ref{lem:4.5} with $s$ as defined above, $t=\omega_0$ and accuracy $\varepsilon$ to obtain a sequence $i_1,\ldots i_k\in [\ell]$ such that with $x_j:=\Phi_{i_j}\circ \ldots\circ \Phi_{i_1}(s)$ we have $x_k\in B(\omega_0,\epsilon)$.
Here we may assume that no $x_j$ is equal to $y_{G_{i_j}}(p)$. 
Otherwise we let $j$ be the last index such that $x_j=y_{G_{i_j}}(p)$ 
replace $s$ by $y_{G_{i_j}}(p)$ and shorten the sequence by letting it start at $j+1$.

Then by~\eqref{eq:composition Phi} $\omega_1:=x_k$ is the effective edge interaction of a two-terminal graph $G\in \Hs_{G_0}$, and satisfies $|\omega_1-\omega_0|\leq \varepsilon$.
More precisely, letting $H_0=G_0^{\parallel n_s}$ and $H_j=(H_{j-1}\parallel G_0)\series G_{i_{j-1}}$, we have $G=H_k$.
We can also compute its reliability polynomial, $R(G;p)$, using Lemma~\ref{lem:recursion} in time $O(k)= \poly(\size(s,t,\epsilon))= \poly(\size(y_0,\epsilon))$. Additionally, $|G|=O(k)=O(\log(\epsilon^{-1}))$  $=\poly(\size(y_0,\varepsilon))$.

It remains to show that $R(G;p)\neq 0$. We in fact claim that $R(H_j;p)\neq 0$ for all $j=0,\ldots,k$.
We prove this by induction, the base being covered since either $H_0=G_0^{\parallel n_s}$, or $H_0=G_{i_\ell}$ for some $\ell$ and in either case we have $R(H_0;p)\neq 0$.
Now suppose that $R(H_j)\neq 0$ for some $j<k$ and assume towards contradiction that $R(H_{j+1};p)=0.$
We have $H_{j+1}=(H_j\parallel G_0)\series G_{i_j}$ and thus by Lemma~\ref{lem:ei-comp} we have $0=R(H_{j+1};p)=R(H_j\parallel G_0;p)R(G_{i_{j+1}};p)$ and hence $R(H_j\parallel G_0;p)=0$.
Since $y_{H_{j+1}}=y_{H_j\parallel G_0}+y_{G_0}-1\neq \infty$, it follows that $S(H_j\parallel G_0;p)=0$.
This implies $S(H_j;p)=0$, but since $R(H_j;p)\neq 0$ it then follows that $x_j=y_{H_j}(p)=1$ contradicting our assumption that no $x_j=1$. 
Indeed if, $x_j=1$, then $x_{j+1}=\Phi_{i_{j+1}}(1)=y_{G_{i_{j+1}}}(p)$ and we assumed to no such index $j+1$ exists.
This shows that $R(H_j;p)\neq 0$ for all $j$.

\item  Take $n=\left\lceil\tfrac{|\omega_0-1|}{r}\right\rceil\in\N$ and note that $n\leq \tfrac{|p|+3}{r}+1$ and hence $n$ is constant in terms of the input. 
Then $u=\tfrac{\omega_0-1}{n}+1\in B(1,r)$. (Also, note that $u\in\mathbb{R}$ in case $p\in 
\mathbb{R}$.)
We then run the algorithm from Case (1) with $s$ as defined above, $t=u$ and accuracy $\varepsilon/n$.
The output $x$ is the effective edge interaction of some $H\in \Hs^*_{G_0}(p)$ and satisfies $|u-x|<\varepsilon/n$. 
The running time is bounded by $\poly(\size(y_0,\epsilon/n))=\poly(\size(y_0,\epsilon))$ since $n$ is constant.

By Lemma~\ref{lem:eei}, $\omega_1:=nx-n+1$ is the effective edge interaction of $H^{\series_n}$ and satisfies
\[
|\omega_1-\omega_0|=|nx-n+1-(nu-n+1)|=n|x-u|<\epsilon.
\]
We output the graph $G=H^{\series_n}$ and the value $R(H^{\series_n};p)=R(H;p)^n\neq 0$ by Lemma~\ref{lem:recursion}. 
Additionally since $n=O(1)$ we have
\begin{equation*}
    |H^{\series_n}|<n|H|=\poly(\size(y_0,\epsilon)).
\end{equation*} 
The running time is also bounded by $\poly(\size(y_0,\epsilon))$. 

\item  
We may assume that $\varepsilon<r/2$ and hence $B(\omega_0,\varepsilon)$ does not contain $1$.

By our pre-computation we know there exists $i\in \{1,\ldots, N\}$ and $x_0\in B(p,2)$ such that $g^{\circ i}(x_0)=\omega_0$. We can compute it by computing the inverse $g^{-1}$ and determine which value $(g^{-1})^{\circ i}(\omega_0)$ lies in $B(p,2)$. This takes only polynomial time in terms of $\size(\omega_0)=O(\size(y_0))$, since $N$ is constant.
The idea is now to get an effective interaction $x_1$ that is close to $x_0$ using Case (2) and then apply $g^{\circ i}$ to it to obtain $\omega_1\in B(\omega_0,\varepsilon)$.
To make this precise, we first need to find out how close exactly we need to get to $x_0$.

Recall that $h$ is the M\"obius transformation defined by $z\mapsto \frac{z-1}{z-p}$, which sends $1$ to $0$ and $p$ to $\infty$.
Then $\hat{g}:=h\circ g\circ h^{-1}$ is given by $z\mapsto \alpha z$ with $\alpha=g'(1)$.
Note that $h^{-1}(z)=\frac{pz-1}{z-1}$.

Let us denote $z'=h(z)$ for $z\in \widehat{\C}$. 
In these new coordinates it is easy to see that if $x_1'$ is such that $|x_1'-x_0'|\leq \eta$ for some $\eta>0$, then with $\omega_1'=\hat{g}^{\circ i}(x_1')$ we have $|\omega_0'-\omega_1'|\leq |\alpha|^{i}\eta\leq \eta$. 
To transfer this to the original coordinates we need to quantify what happens under the maps $h$ and $h^{-1}$.

Starting with $\omega_0-\omega_1$, we have by definition,
\begin{align*}
\omega_0-\omega_1=h^{-1}(\omega_0')-h^{-1}(\omega_1')=\frac{(\omega_0'-\omega_1')(1-p)}{(\omega_0'-1)(\omega_1'-1)}.    
\end{align*}
Therefore, if 
\begin{align}\label{eq:minimum}
|\omega_0'-\omega_1'|\leq\min\left\{\frac{\varepsilon|\omega_1'-1|^2}{2|p-1|}, |\omega_0'-1|\right\},    
\end{align} 
it follows that $|\omega_0-\omega_1|\leq \varepsilon$ (here we use $|\omega_1'-1|\leq |\omega_1'-\omega_0'|+|\omega_0'-1|$).

Let us next denote $\varepsilon'$ for the minimum in~\eqref{eq:minimum}.
It thus suffices to have $|x_0'-x_1'|\leq\varepsilon'$.
We have
\begin{align*}
x_0'-x_1'=h(x_0)-h(x_1)=\frac{(x_0-x_1)(p-1)}{(x_0-p)(x_1-p)}.  
\end{align*}
Therefore, if 
\begin{align*}
|x_0-x_1|\leq\min\left\{\frac{\varepsilon'|x_0-p|^2}{2|p-1|}, |x_0-p|\right\},
\end{align*} 
it follows that $|x_0'-x_1'|\leq \varepsilon'$. 

To summarize, if $|x_1-x_0|$ is smaller than
\begin{align*}
\varepsilon'':=\min\left\{ \frac{\varepsilon |\omega_0'-1|^2|x_0-p|^2}{4|p-1|^2}, \frac{\varepsilon |\omega_0'-1|^2|x_0-p|}{2|p-1|}, \frac{\varepsilon |x_0-p|\cdot |\omega_0'-1|}{2|p-1|}, {|\omega_0'-1||x_0-p|} \right\},
\end{align*}
then $|\omega_0-\omega_1|\le \varepsilon$.

Next we claim that
\begin{equation}\label{eq:size claim}
    \size(\varepsilon'')=O(\size(\varepsilon,y_0)).
\end{equation}
To see this first note that $x_0=(g^{-1})^{\circ i}(y_0)$ and $\omega_0'=h(\omega_0)$, where $h$ and $(g^{-1})^{\circ i}$ are both M\"obius transformations with constant coefficients (i.e. not depending on $y_0$ nor $\varepsilon$).
Therefore there are constant algebraic numbers $a,a',b,b',c,c',d,d'$ such that
\begin{equation*}
 x_0-p=\frac{a\omega_0+b}{c\omega_0+d} \quad \text{and}    \quad \omega_0'-1=\frac{a'\omega_0+b'}{c'\omega_0+d'}.
\end{equation*}
It follows from standard facts about algebraic numbers (cf.~\cites{Waldschmidtbook} and more specifically~\cite{main-approx}*{Lemma 3.7}) that the absolute logarithmic heights of $x_0-p$ and $\omega_0'-1$ are bounded by $O(\size(\omega_0))=O(\size(y_0))$.
Since $x_0\neq p$ and $\omega_0'\neq 1$ it follows from~\cite{main-approx}*{Lemma 3.7} that
\begin{equation*}
\log(|x_0-p|)=O(\size(y_0)) \quad \text{and}\quad \log(|\omega_0'-1|)=O(\size(y_0)).
\end{equation*}
This implies that $\size(\varepsilon'')=O(\size(\varepsilon,y_0))$.

By the algorithm guaranteed from Case (2) we now find $x_1$ as the effective edge interaction of some two-terminal graph $H\in \Hs^*_{G_0}(p)$ such that $|x_1-x_0|\leq \varepsilon''$. 
Since $\size(x_0)=O(\size(y_0))$ the size of $H$ and the running time of the algorithm are both bounded by $\poly(\size(y_0,\varepsilon''))$.
By applying $g^{\circ i}$ to $x_1$ we find $\omega_1$ such that $|\omega_0-\omega_1|\leq \varepsilon$ and such that $\omega_1$ is the effective edge interaction of $G:=H\parallel G_0^{ \parallel i}$.
We can compute $R(G;p)$ using Lemma~\ref{lem:recursion} in $\poly(\size(y_0,\varepsilon''))$ time having access to $R(H;p)$ from the output of Case (2).
Since $\size(\varepsilon'')=O(\size(\varepsilon,y_0))$.
It follows that both the size of $H\parallel G_0^{\parallel i}$ and the running time of the algorithm are bounded by $\poly(\size(y_0,\varepsilon))$, as desired.

It remains to argue that $R(G;p)\neq 0$. Suppose towards contradiction that $R(G;p)=0$.
Then, since $\omega_1=y_G(p)\neq \infty$, 
it must be that $S(G;p)=0$.
From this it follows by Lemma~\ref{lem:recursion} that $S(H;p)=0$ since $S(G_0;p)\neq 0$ by assumption.
This implies that $x_1=y_H(p)=1$, as $R(H;p)\neq 0$ since $H\in \Hs^*_{G_0}(p)$.
But then $\omega_1=g^{\circ i}(1)=1$, contradicting our assumption that $1\notin B(\omega_0,\varepsilon)$.
This finishes the proof of Case (3).
\end{enumerate}
Since all three cases have been covered, this finishes the proof.
\end{proof}

\subsection{Computing ratios}\label{sec:ratio computing}
Here we give an algorithm to (essentially) compute the ratio $r=\frac{R(F;p)}{R(F\setminus e;p)}$ as the (shifted) root $y^*$ of the equation $R(F;p)+(y-(1+p))R(F\setminus e;p)=0$. 
We start with a lemma that says that given an oracle for \textsc{$p$-Abs-Rel} (resp. \textsc{$p$-Arg-Rel} we can approximate $R(F;p)+(y-(p+1))R(F-e;p)$.
\begin{lemma}\label{lem:step 1 in the reduction}
Let $p\neq 1$ be an algebraic number.
Suppose there exists an algorithm that on input of a (planar) graph $H$ computes a $0.25$-abs approximation (resp. $0.25$-arg approximation) to $R(H;P)$ in time polynomial in $|H|$.

Then there exists an algorithm that on input of a (planar) graph $H$ and an edge $e$ of $H$, a two terminal graph $G\in \Hs^*_{G_0}(p)$, and the number $R(G;p)$ that computes a $0.25$-abs approximation (resp. $0.25$-arg approximation) to $R(H;p)+(y_G-(p+1))R(H\setminus e;)$ in time polynomial in $(|H|+|G|)$. In the planar setting we assume $G_0$ is planar with its terminals on the same face.
\end{lemma}
\begin{proof}
By assumption $R(G;p)\neq 0$ and therefore by Lemma~\ref{lem:eei} we have 
\begin{align}
\frac{1-p}{R(G;p)}R(H(G)_e;p)=R(H;p)+\left(y_G(p)-(p+1)\right)R(H\setminus e;p).\label{eq:linear target}
\end{align}
We can thus use the assumed algorithm to compute a $0.25$-abs approximation (resp. $0.25$-arg approximation) to $R(H(G)_e;p)$ in time polynomial in $|H|+|G|$, from which we derive the required approximation to $R(H;p)+(y_G-(p+1))R(H\setminus e;p)$ by multiplying the result by the exact quantity $\frac{1-p}{R(G;p)}$ (since by assumption $R(G;p)\neq 0$).
Note that in the planar setting we have that $G$ is planar with its two terminals on the same face as follows by an easy induction. Therefore $H(G)_e$ is planar and we can thus indeed use the assumed algorithm.
\end{proof}

An important ingredient is the \emph{box shrinking} procedure from~\cite{main-approx} captured as Theorem 3.4 in there. Technically, as stated in~\cite{main-approx} it cannot be used and we therefore slightly adjust the statement below, noting that the proof of the theorem given in~\cite{main-approx} actually gives the statement below. 

\begin{theorem}\cite{main-approx}*{Theorem 3.4}\label{thm:box-shrinking}
    Let $A, B$ be complex numbers and let $C > 1$ be a rational number, such that $|A|$ and $|B|$ are both at most $C$, and both are either $0$ or at least $1/C$. Assume one of the following:
    \begin{itemize}
        \item there exists a $\poly(\log C,\size(y_0,\epsilon))$-time algorithm to compute, on input of $y_0\in\Q[i]$ and a rational number $\epsilon>0$, an $0.25$-abs-approximation of $A\hat{y}+ B$ for some algebraic number $\hat{y}\in B(y_0,\epsilon)$, or,
        \item there exists a $\poly(\log C,\size(y_0,\epsilon))$-time algorithm to compute, on input of $y_0\in\Q[i]$ and a rational number $\epsilon>0$, an $0.25$-arg-approximation of $A\hat{y}+ B$ for some algebraic number $\hat{y}\in B(y_0,\epsilon)$.
    \end{itemize}
Then, there exists an algorithm that, on input of a rational $\delta>0$ and $C > 0$ as above, outputs \airquotes{$A=0$} when $A = 0$ and $B\neq 0$, and that outputs \airquotes{$A\neq 0$} and a number $y\in\Q[i]$, such that $-B/A\in B_\infty(y, \delta/2)$ when $A\neq 0$. When $A = B = 0$ it is allowed to output anything. The running time and the size of $y$ are both $\poly(\log(C/\delta))$.
\end{theorem}
It is useful to think of $A=R(H\setminus e;p)$ and $B=R(H;p)$ in the above theorem.
To use this theorem above we will need some (height) bounds on the reliability polynomial.
\begin{lemma}\label{lem:height bounds}
    Let $G=(V,E)$ be a graph $m$ edges with edge $e$ and let $p$ be an algebraic number of degree $d=d(p)$ and absolute logarithmic height $h=h(p)$. Then
    \begin{itemize}
        \item[(a)] If $R(G;p)\neq 0$, then \[|\log (|R(G;p)|)|\leq dm\left((\log(4)+h(p)\right)).\]
        \item[(b)] If $R(G\setminus e;p)\neq 0$, then \[\log(H\left(\frac{R(G;p)}{R(G\setminus e;p)}\right)\leq 2dm(\log(4)+h(p)).
        \]
    \end{itemize}
\end{lemma}
\begin{proof}
To prove (a) note that since $R(H;p)=\sum_{F\subseteq E} p^{|E\setminus F|}(1-p)^{|F|}$, where the sum runs over all connected sets, it follows that $R(H;p)$ is a polynomial in $p$ of degree at most $m$ and where the sum of the absolute values of coefficients is bounded by $2^{2m}$.
It follows from~\cite{main-approx}*{Lemma 3.7(b)} that $h(R(H;p))\leq m(\log(4)+h(p)).$
Consequently,~\cite{main-approx}*{Lemma 3.7(a)} implies $|\log (|R(H;p)|)|\leq dm(\log(4)+h(p))$, using that $d$ is an upper bound on the degree of $R(H;p)$ as an algebraic number.

The proof of (b) follows along the same lines. We have by (a) and by~\cite{main-approx}*{Lemma 3.7(b)}, $h(\frac{R(G;p)}{R(G\setminus e;p)})\leq 2m(\log(4)+h(p))-\log(4)-h(p)$.
The claimed bound then follows from~\cite{main-approx}*{Lemma 3.7(c)} using that the ratio has degree at most $d$.
\end{proof}

We are now ready to state and prove a theorem that allows us to (essentially) compute the ratio $R(F;p)/R(F-e;p)$ if we have access to an oracle for \textsc{Approx-Abs-Planar-Rel$(p)$} or  \textsc{Approx-Arg-Planar-Rel$(p)$}.

\begin{theorem}\label{thm:compute ratio}
Let $G_0$ be a two-terminal graph and let $p$ be a non-positive algebraic number contained in $\mathcal{D}_{G_0}\cup\mathcal{D}^{\mathbb{R}}_{G_0}$.
Suppose there exists and algorithm that on input of a graph $F$ computes a $0.25$-abs approximation (resp. $0.25$-arg approximation) to $R(F;p)$. 
Then there exists an an algorithm that on input of a graph $F$ and an edge $e$ outputs in polynomial time in $|F|$ an algebraic number $r$ and a bit $b\in \{0,1\}$ satisfying the following: 
\begin{itemize}
    \item[(1)] if $R(F-e;p)\neq 0$, then $b=1$ and $r=\frac{R(F;p)}{R(F-e,p)}$;
    \item[(2)] if $R(F-e;p)=0$ and $R(F/e;p)\neq 0$, then $r=1-p$ and $b=0$;
    \item[(3)] if both $R(F-e;p)=0$ and $R(F/e;p)=0$, then the algorithm may output any algebraic number $r$ and bit $b$.
\end{itemize}
Moreover, if $G_0$ is planar with its terminals on the same face, then the graphs $F$ can be restricted to being planar.
\end{theorem}
\begin{proof}
To prove this we follow part of the proof of~\cite{main-approx}*{Theorem 3.12}.
Let us write $B:=R(F;p)$ and $A:=R(F-e;p)$ and interpret the ratio 
$-R(F;p)/R(F-e;p)$ as a shifted root of the equation $Ax+B=0$, where we think of $x=y_G-(p+1)$ for some $G\in \Hs^*_{G_0}(p)$.
Let us denote $d=d(p)$ and $g=h(p)$ (the degree and absolute logarithmic height of the algebraic number $p$).

Let $C$ be the smallest integer bigger than $\exp(dm(\log(4)+h(p))$, where $m$ denotes the number of edges of $F$.
Then by Lemma~\ref{lem:height bounds}, both $|A|<C$ and $|B|<C$ as well as $|A|>1/C$ if $A\neq 0$ and $|B|>1/C$ if $B\neq 0$.
We know show that Theorem~\ref{thm:constructing} combined with Lemma~\ref{lem:step 1 in the reduction} gives us the desired algorithm to be able to apply the box shrinking procedure (Theorem~\ref{thm:box-shrinking}).
Indeed on input of $y_0\in \Q[i]$ and $\varepsilon>0$ the algorithm of Theorem~\ref{thm:constructing} gives us a two-terminal graph $G\in \Hs^*_{G_0}(p)$ such that $|y_0-(y_{G}-(p+1))|\leq \varepsilon$ in time $\poly(\size(y_0,\varepsilon))$ as well as the value $R(G;p)$.
Applying the algorithm of Lemma~\ref{lem:step 1 in the reduction} we obtain an $0.25$-abs-approximation (resp. $0.25$-arg-approximation) to $R(F;p)+(y_{G}-(p+1))R(F-e;p)$ in time $\poly(|G|+|F|)=\poly(\log(C),\size(y_0,\varepsilon))$.
We can therefore apply the box shrinking procedure (Theorem~\ref{thm:box-shrinking}) with \[\delta=\exp(-(d^2+5d))C^{-2},\]
which has running time bounded by $\poly(\log(C/\delta))=\poly(\log(C))$.

If the output of that algorithm is $"A"=0$ we output, $b=0$ and $r=1-p$ and if it outputs $"A\neq 0"$ and $y\in \mathbb{Q}[i]$ we output $b=1$ and we use a slight modification of the  LLL~\cite{LLL}-based algorithm of Kannan, Lenstra and Lov\'asz~\cite{kannan-1988} in the form of~\cite{main-approx}*{Proposition 3.9}. We input $d$, $H=2C$ and $y$ to that algorithm and in time $\poly(d, \log(C),\size(y))=\poly(\log(C))$ it outputs a polynomial $q$.
We then output the algebraic number given by $q$ and the box $B_\infty(y,\delta/2)$.
The overall running time is bounded by $\poly(\log(C))=\poly(\size(F))$.

By Theorem~\ref{thm:box-shrinking} it suffices to show that if $R(F-e;p)\neq 0$, then $q$ is the minimal polynomial of $-R(F;p)/R(F-e;p)$ and that $B_\infty(y,\delta/2)$ contains no other zeros of $q$.
By Lemma~\ref{lem:height bounds} we know that the height of $-R(F;p)/R(F-e;p)$ is at most $2C=H$ and that it has degree at most $d$. 
Thus by~\cite{main-approx}*{Proposition 3.9} $p$ is indeed the minimal polynomial of $-R(F;p)/R(F-e;p)$. 
Finally, it follows by a result of Mahler~\cite{Mahler64} the absolute value of the logarithm of the distance between any two distinct zeros of $q$ is at most $3d/2\log(d)+d\log(H))<\log(1/\delta)$ and hence $-R(F;p)/R(F-e;p)$ is the only zero of $q$ contained in $B_\infty(y,\delta/2)$ and therefore the pair $(q,B_\infty(y,\delta/2))$ forms a representation of the algebraic number $-R(F;p)/R(F-e;p)$, as desired.

In case $G_0$ is planar with its two vertices of on the same face, the same is true for any $G\in \mathcal{H}^*_{G_0}$ and we can indeed restrict our graphs $F$ to be planar in the statement without affecting the conclusion.
\end{proof}

\subsection{Telescoping}\label{sec:telescoping}
Here we adapt the telescoping procedure from~\cite{main-approx} to the setting of the reliability polynomial. After which we combine it with Theorem~\ref{thm:compute ratio} to finally finish the proof of Theorem~\ref{thm:main hard}.
\begin{theorem}\label{thm:telescoping}
Let $p\notin \{0,1\}$ be a fixed algebraic number and suppose that we have an algorithm that on input of a (planar) graph $G$ and an edge $e$ of $G$ outputs an algebraic number $r$ and a bit $b\in\{0,1\}$ in polynomial time in $|G|$ such that
\begin{itemize}
    \item[1] if $R(G\setminus e;p)\neq 0$ then $b=1$ and $r=\frac{R(G;p)}{R(G\setminus e,p)}$;
    \item[2] if $R(G\setminus e;p)=0$ and $R(G/e;p)\neq 0$ then $r=1-p$ and $b=0$;
    \item[3] if both $R(G\setminus e;p)=0$ and $R(G/e;p)=0$ then the algorithm may output any algebraic number $r$ and bit $b$.
\end{itemize}
Then there is an algorithm that on input of a (planar) graph $G$ computes $R(G;p)$ in $\poly(|G|)$ time.
\end{theorem}
\begin{proof}
    We construct a sequence of graphs $G_0,\ldots,G_m$ where $G_0=G$ and $G_m$ is a graph with no edges as follows: for $i\geq 0$, we apply the assumed algorithm to the graph $G_i$ and an edge $e_i$ of $G_i$. The algorithm outputs the pair $(r_i,b_i)$; if $b_i=1$ we choose $G_{i+1}=G_i\setminus e_i$, otherwise we choose $G_{i+1}=G_i/e_i$, and if $G_{i+1}$ has no edges the algorithm terminates and we set $m=i+1$.
 We output $c\prod_{i=0}^{m-1} r_i$, where $c= R(G_m;p)$. 
  Since $G_m$ has no edges, $c=1$ if and only if $G_m$ has exactly one vertex, otherwise $c=0$. 

Since the running time is clearly polynomial in $|G|$ it suffices to prove that
    $$R(G;p)=c\prod_{i=0}^{m-1}r_i.$$
Let us first assume that $R(G;p)\neq 0$.
We show by induction that $R(G_i;p)\neq 0$ for all $i=0,\ldots,m$.
Assuming $R(G_i;p)\neq 0$ we know that not both $R(G_i-e_i;p)=0$ and $R(G_i/e_i;p)=0$. In case $R(G_i\setminus e_i;p)\neq 0$ the algorithm outputs $b_i=1$ and hence $G_{i+1}=G_i-e_i$, and thus $R(G_{i+1};p)\neq 0$. In case $R(G_i-e_i;p)=0$ and hence $R(G_i/e_i;p)\neq 0$ the algorithm outputs $b_i=0$ and hence $G_{i+1}=G_i/e_i$, and thus also in this case $R(G_{i+1};p)\neq 0$.
Since $G_m$ has no edges, $R(G_m;p)=c$ for all $p\in\C$, where $c=1$ if and only if $G_m$ has exactly one vertex, otherwise $c=0$. 
It thus follows that $c=1$. 
Since in case $b_i=0$ we have $R(G_i-e_i;p)=0$ and thus by deletion contraction (Lemma~\ref{lem:con-del})  $R(G_i;p)=(1-p)R(G_i/e_i;p)=(1-p)R(G_{i+1};p)$, it follows that $r_i=\frac{R(G_i;p)}{R(G_{i+1};p)}$ for all $i$. Therefore 
\[
R(G;p)=\prod_{i=0}^{m-1}\frac{R(G_i;p)}{R(G_{i+1};p)}=c\prod_{i=0}^{m-1}r_i,
\]
as desired.

In case $R(G;p)=0$ we must show that $c\prod_{i=0}^{m-1}r_i=0$.
We may assume that $G_m$ consists of single vertex since otherwise $c=0$ and we are done.
Then $G_m=K_1$ and $R(G_m;p)\neq 0.$
Let $i$ be the first index for which  $R(G_i;p)\neq 0.$
Then $i>0$ and $R(G_{i-1};p)=0$.
It must be the case that $R(G_{i-1}\setminus e_{i-1};p)\neq 0$ for otherwise $R(G_{i-1}/e_{i-1};p)=0$ by the deletion contraction recurrence (and the fact that $p\notin \{0,1\}$), in which case $R(G_i;p)$ would equal $0$.
Then the algorithm outputs $b_i=1$ and $r_i=R(G_{i-1};p)/R(G_i\setminus e_i;p)=0$ and thus the product of the $r_j$ is equal to $0$, as desired.

Clearly, if $G_0$ is planar, any of the graphs $G_i$ is also planar and thus the algorithm is also correct in the planar setting.
\end{proof}

We can now prove Theorem~\ref{thm:main hard}.
\begin{proof}[Proof of Theorem~\ref{thm:main hard}]
We focus on the case that $G$ is planar and has its terminals on the same face.
The general case goes along exactly the same lines.

Suppose there exists and algorithm that on input of a graph $F$ computes a $0.25$-abs approximation (resp. $0.25$-arg approximation) to $R(F;p)$.
Then Theorem~\ref{thm:compute ratio} combined with Theorem~\ref{thm:telescoping} gives an algorithm that on input of a planar graph $F$ computes $R(F;p)$ exactly in time polynomial in $|F|$.

Since by Vertigan's result (Theorem~\ref{thm:vertigan}) \textsc{Planar-Rel}$(p)$ is \textsc{\#P}-hard, it follows that both \textsc{Approx-Abs-Planar-Rel$(p)$} and \textsc{Approx-Arg-Planar-Rel$(p)$} are \textsc{\#P}-hard.
\end{proof}

\section{Concluding remarks}
We collect some concluding remarks here.

While our framework of systematically exploiting series and parallel compositions of two-terminal graphs has shed new light on the question of whether the set of reliability zeros is bounded, we have unfortunately not been able to settle this question. 
Proposition~\ref{prop:unbounded zeros?} provides two equivalent versions of this question, neither of which make it more clear what the truth should be.

Another natural question that arises from our work is: can the bound of $k\le 4$ in Proposition~\ref{prop:zeros outside disk} be improved? Recall that if we don't have any constraint on $p^k$ in Proposition~\ref{prop:zeros outside disk}, then this would imply that the collection of reliability zeros is dense in the complex plane.

A closely related dual object to the reliability polynomial is the forest generating function $F_G(x)=T_G(x+1,1)$. 
For planar graphs the forest generating function is (up to a reparametrization) equal to the reliability polynomial of the planar dual of the graph.
As such it would be interesting to study the zeros of the forest generating function of graph families. 
Perhaps they shed some light on the question whether or not reliability zeros are dense in the complex plane or not. 
We expect that statements similar to Theorems~\ref{thm:main equal} and~\ref{thm:main hard} are true for the forest generating function, but we leave the details open for further work. 

\bibliographystyle{abbrv}
\bibliography{rel}
\clearpage
\section*{Appendix}
\hypertarget{code:pentagon}{We} present a proof of Lemma~\ref{lemma:pentagonal} in the form of Sage code~\cite{sagemath} which directly computes $\hat{y}_{H(G_1,G_2)}(p)$ and $\abs{F(e^{it},e^{-it})}^2$; it can be run directly by clicking \href{https://sagecell.sagemath.org/?z=eJytVE1v2zAMvRvwfxCwg6XWcSO1uxTwoRnQdOfulrWG4iiOUFkKJDmdse6_j5KdfncItukgM6TeE_nIaMctznqa9ywjaZImn74Y7bztao_8RqCt0J43RqPG8u0GzdG99Buk-FIoJVZIrBrh0mRezkMYL_A0pySHncX9FHYabTraZ7CzaLNon8J-AxdzWi4-hD07CidZub_miQ78a2PRHZIacXqeJgjWvHDCVyHFKiaM7xbTm_xuQW_yjNNQ7hOGHYLZS3S9MfdRnSgKlF848GDFe9P5MqulrTvFbZY_8bjym-1EXhtlbLXsB2f5M3CeZ9603Jvs18Bem3bbee4lqG7W8Z62U17uuJXcC2SFknwplfQ92hrVa9NKrsLReZpsKwiX06G0i1DadbeEkhzmlIw1htDsRYjtQ2FJIBq6qaTz2HLdCPyZkDz-vCDH8TuDthXSVbXRWtRerPBzjrBiKsclphN60lNyq4QG9FFPb_FgTsIH0iJH318ixzUg2YCcAZINyNmIZISMBRfrTqnKyXar5LrHUce1NS1yvBFF7JIrxkzlLggHR4316DH7KqhuNEx7qMhzqaVuqp2wXvz4ey5X6a5dCvtRXx3k6ycfNxTP8ynMeJq4_9DXq_LAtqIXs_CHsvAVKUv2qu0AOUPa-JDIQfLiqxxNXw9PWO6fJ-jwMdpwX_VlvPBkvJfCcFmpPc5i8DzL47eA2oWF_6vF5PXc5W-yyE6yt86BZyVCo_X7TO8-BQ-XD7cMgR1mJ7wySqTJLjzgPrxMl6wcmB30H_e0rI3Dnhx_PXJSgwFv_Oia7F174d7DTd7iHqnIozYxJ9DmkhX77CtvZQMl_AZv09-w&lang=sage&interacts=eJyLjgUAARUAuQ==}{here}.
\begin{lstlisting}
var('y1,y2')

#Construct the pentagon graph G with labelled edges
G=Graph([(0,1),(0,2),(0,3),(1,2),(1,3),(1,4),(2,3),(2,4),(3,4)])
a1=[(0,2),(0,3),(1,2),(1,3),(2,4),(3,4)]
a2=[(0,1),(1,4),(2,3)]
for k in a1:
    G.set_edge_label(k[0],k[1],'a1')
for k in a2:
    G.set_edge_label(k[0],k[1],'a2')

#Show the graph
G.show(layout='circular',edge_labels=True,color_by_label={'a2':'tomato'})

#computation of the multivariate reliability polynomial of G
p_rel=0
for A in Subsets(a1):
    for B in Subsets(a2):
        if Graph([list(range(5)),list(A)+list(B)]).is_connected():
            p_rel+=(1-1/y1)^len(A)*y1^(len(A)-len(a1))*\
                   (1-1/y2)^len(B)*y2^(len(B)-len(a2))
p_rel.full_simplify()

from sage.graphs.connectivity import connected_component_containing_vertex
from sage.graphs.connectivity import connected_components_number

#computation of the split-reliability polynomial of (G,0,4)
sp_rel=0
for A in Subsets(a1):
    for B in Subsets(a2):
        H=Graph([list(range(5)),list(A)+list(B)]) 
        if connected_components_number(H)==2:
            if 4 not in connected_component_containing_vertex(H, 0):
                sp_rel+=(1-1/y1)^len(A)*y1^(len(A)-len(a1))*\
                        (1-1/y2)^len(B)*y2^(len(B)-len(a2))
hat_y=p_rel/sp_rel+1
print('hat_y:',hat_y.numerator().full_simplify(),
               '/',
               hat_y.denominator().full_simplify())

#computation of |F|^2 on the circle
var('t')
F2=hat_y.subs(y1=cos(t)+I*sin(t),y2=cos(t)-I*sin(t))*\
   hat_y.subs(y1=cos(t)-I*sin(t),y2=cos(t)+I*sin(t))
print('|F|^2:',F2.simplify_trig())
\end{lstlisting}
\clearpage
\hypertarget{code:5to9}{We} present a proof of Proposition~\ref{prop:5to9} in the form of Sage code~\cite{sagemath} which computes the virtual edge interaction and $\gcd(R(G_k;p),p^k-1)$ of the graphs represented by the given adjacency matrices; it can be run directly by clicking \href{https://sagecell.sagemath.org/?z=eJydk11vmzAUhu8r9T9Y2gV2Qgi0-VonrjaJ7KLStO0uSiMH3MQKwZ4xaUnb_z5_QNMUkkrFEoFzjl8_ec8hCZ8uL4C6voIbcIuloI9wZiP6mvmu7wZqTebuh1Ed8d2Je9WITlT8ODrRdWr58_khiqqKyXmWcSvLuIVl3MIybrCMdZ1hOQTnNcv4PMuolWXUwjJqYRk1WEa67hTL6DxL4A5bad7HLc9Qr0Z8aDJN0qHJHMeNwmna4Ue0gxO0gxbagVpNqoHJtMV15voobhTUun6lnaPLi5fLi3smwAbQDCQ3Nv5FvVBJcUpzmq1AzgoRk77EYkUk2BEhaUxygLME7Bc-2OG0IHZjHvr2QYaBfdj7IXnk8KrDaednf4NslAuaSXhgc7ROCMgdvAKcAgqCvtPNpYAb1HWQYwurvTViynCi6VYC8zW4F2wLtmTLRGnAeMqk1Gkq7YYojHQhTGabeaX0K4w8XQdTXLJChk5MRVykWDh1gZev2QN8d3DMtrww2nJNgCApxUuaUlkCztIyY1tlnOXR-b-FlORNxqrssIAOr8_hC6USvvEj8qTetjhsqyH05eXFMoePYeCWYdDnqAODHkd3MPKYSIiAqBegDtfvOd0T9fqa6AboyMtGH6LvP-BvGC023zhylerdBqm-ON1Die6K4fVUX5XTEHmrOIG2FrV26tiwnCuvTtm2LIE-PpcoVL-o-0fdapmHNREE6CSguZGyrWdLiWlGEr2ZJiST9L40Y2sGQerBjqzGVDU8ZryszZx6W6JmelGPNJzlrqynI8ttW8Abf6af6sv0XV-mdV-mp_pSHV0h9Mz9nKc7KmSh7CPJiqj_K4nAsaQss9VrLMtFFBqVvpXsBsdyyoHKNamqq8H-V1ChbOWCcW2Q-uTXLE1OTM6z2vdUviz08CCovmf03DI5WBlkcaxZPNz7CHkZRB8M5rOdynO6dioPska1TdbpffJyjmTQf4Af790=&lang=sage&interacts=eJyLjgUAARUAuQ==}{here}.
\begin{lstlisting}
d={
    9 : Matrix([
        [0,0,1,1,8],
        [0,0,1,1,8],
        [1,1,0,8,2],
        [1,1,8,0,2],
        [8,8,2,2,0]]
        ),
    8 : Matrix([
        [0,0,1,1,7],
        [0,0,1,1,7],
        [1,1,0,7,2],
        [1,1,7,0,2],
        [7,7,2,2,0]
        ]),
    7 : Matrix([
        [0,0,1,1,6],
        [0,0,1,1,6],
        [1,1,0,6,2],
        [1,1,6,0,2],
        [6,6,2,2,0]
        ]),
    6 : Matrix([
        [0,0,1,1,1,5],
        [0,0,1,1,1,5],
        [1,1,0,5,5,5],
        [1,1,5,0,5,2],
        [1,1,5,5,0,2],
        [5,5,5,2,2,0]
        ]),
    5 : Matrix([
        [0,0,1,1,1,4],
        [0,0,1,1,1,4],
        [1,1,0,4,4,2],
        [1,1,4,0,4,2],
        [1,1,4,4,0,3],
        [4,4,2,2,3,0]
    ])
}
for k in d:
    # initialising source/target vertices and z_0 value
    s=0
    t=1
    z0=exp(2*pi*I/k)
    print(
        'z_0 = e^(2 pi i 1/'+str(k)+')'
    )
    
    # loading graph from memory and plotting it
    G=Graph(d[k])
    P=G.plot(layout='circular')
    P.show()
    
    # computing the reliability polynomial from the Tutte polynomial
    var('p')
    p_rel=(
        G.tutte_polynomial()
        .subs(x=1,y=1/p)*(1-p)^(G.order()-1)*p^(G.size()-G.order()+1)
    )
    print(
        'GCD(R(G_k;p),1-p^k) = '+
        str(p_rel.expand().gcd(1-p^k))
    )
    
    # computing the split reliability polynomial by R(G_st)=R(G)+S(G)
    # where G_st is the graph obtained by identifying s and t in G
    H=G.copy()
    H.merge_vertices([s,t])
    nsp_rel= (
        H.tutte_polynomial()
        .subs(x=1,y=1/p)*(1-p)^(H.order()-1)*p^(H.size()-H.order()+1)
    )
    sp_rel=nsp_rel-p_rel
    
    # computing the virtual edge interaction
    haty_G=p_rel/sp_rel+1
    
    # verifying that the required properties hold
    print(
        '|hat{y}_(G_k)(z_0)| = '+
        str(abs(haty_G.subs(p=z0)).n())
    )
    print(
        '|R(G_k;z_0)| = '+
        str(abs(p_rel.subs(p=z0).n()))
    )
    print('--------------------------------------------------------')
    print()
\end{lstlisting}

\end{document}